\documentclass[10pt, a4paper, reqno]{amsart}
\usepackage{amsthm,amsmath, amsfonts, amssymb, amscd,enumerate,enumitem, amscd, esint, xcolor, url, mathrsfs}
\usepackage[colorlinks]{hyperref}
\usepackage[utf8]{inputenc}
\usepackage{mathtools}

\newtheorem{thm}{Theorem}[section]

\newtheorem{prop}[thm]{Proposition}
\theoremstyle{definition}

\theoremstyle{remark}

\newcommand{\hs}{\mathbb{R}^d_+}
\newcommand{\loc}{\textup{loc}}
\newcommand{\glob}{\textup{glob}}
\newcommand{\RH}{\textup{RH}}
\DeclareMathOperator{\Real}{Re}

\newcommand*\mathinhead[2]{\texorpdfstring{$#1$}{#2}}

\numberwithin{equation}{section}
\usepackage[defaultlines=4, all]{nowidow}
\setnoclub[2]
\allowdisplaybreaks[3]

\setlist[enumerate,1]{label=(\alph*)}

\newcounter{BDQ}

\begin{document}

    \title[Variation operators associated with semigroups generated by...]
    {Variation operators associated with semigroups generated by Hardy operators involving fractional Laplacians in a half space}

    \author[J. J. Betancor]{Jorge J. Betancor}
    \address{Jorge J. Betancor\newline
        Departamento de Análisis Matemático, Universidad de La Laguna,\newline
        Campus de Anchieta, Avda. Astrofísico Sánchez, s/n,\newline
        38721 La Laguna (Sta. Cruz de Tenerife), Spain}
    \email{jbetanco@ull.es}

    \author[E. Dalmasso]{Estefanía Dalmasso}
    \address{Estefanía Dalmasso, Pablo Quijano\newline
        Instituto de Matemática Aplicada del Litoral, UNL, CONICET, FIQ.\newline Colectora Ruta Nac. Nº 168, Paraje El Pozo,\newline S3007ABA, Santa Fe, Argentina}
    \email{edalmasso@santafe-conicet.gov.ar, pquijano@santafe-conicet.gov.ar}

    \author[P. Quijano]{Pablo Quijano}
    
    \date{\today}


    \begin{abstract}
    We represent by $\{W_{\lambda, t}^\alpha\}_{t>0}$ the semigroup generated by $-\mathbb L^{\alpha}_\lambda$, where $\mathbb L^{\alpha}_\lambda$ is a Hardy operator on a half space. The operator $\mathbb L^{\alpha}_\lambda$ includes a fractional Laplacian and it is defined by 
    \begin{equation*}\mathbb L^{\alpha}_\lambda=(-\Delta)^{\alpha/2}_{\hs}+\lambda x_d^{-\alpha}, \quad \alpha\in (0,2], \lambda\geq 0.\end{equation*}
    We prove that, for every $k\in \mathbb N$, the $\rho$-variation operator $\mathcal{V}_\rho\left(\left\{t^k\partial_t^k W_{\lambda,t}^\alpha\right\}\right)$ is bounded on $L^p(\hs, w)$ for each $1<p<\infty$ and $w\in A_p(\hs)$, being $A_p(\hs)$ the Muckenhoupt $p$-class of weights on $\hs$.
    \end{abstract}
    \maketitle

\section{Introduction}

In this paper we consider the non-local Hardy type operator $\mathbb L^{\alpha}_\lambda$ defined, in a formal way, by
\begin{equation*}\mathbb L^{\alpha}_\lambda:=(-\Delta)^{\alpha/2}_{\hs}+\lambda x_d^{-\alpha}, \end{equation*}
on $\hs:=\{x=(x_1, \dots, x_d)\in \mathbb R^d: x_d>0\}$. Here, $\alpha\in (0,2]$ and $\lambda\geq \lambda_*$ where
\begin{equation*}\lambda_*=-\frac{\Gamma(\frac{\alpha+1}{2})}{\pi} \left(\Gamma\left(\frac{\alpha+1}{2}\right)-\frac{2^{\alpha-1}\sqrt{\pi}}{\Gamma\left(1-\frac\alpha 2\right)}\right).\end{equation*}

In order to give a precise definition of the operator $\mathbb L^{\alpha}_\lambda$ we consider the following quadratic forms. We define, for every $u,v\in C^1_c(\hs)$ (the space of continuously differentiable functions having compact support on $\hs$), when $\alpha\in (0,2)$,
\begin{equation*}Q_\lambda^\alpha(u,v)=\frac12 \mathcal A(d,\alpha)\int_{\hs\times \hs} \frac{(u(x)-u(y))\overline{(v(x)-v(y))}}{|x-y|^{d+\alpha}}\, dxdy+\lambda \int_{\hs} \frac{u(x)\overline{v(x)}}{x_d^\alpha}\, dx,\end{equation*}
where $\mathcal A(d,\alpha):=\frac{-\alpha}{2^{1+\alpha}\pi^{d/2}}\frac{\Gamma((d-\alpha)/2)}{\Gamma(1+\frac\alpha 2)}$, and, when $\alpha=2$,
\begin{equation*}Q_\lambda^2(u,v)=\int_{\hs} \nabla u(x)\overline{\nabla v(x)}\, dx+\lambda \int_{\hs} \frac{u(x)\overline{v(x)}}{x_d^2}\, dx.\end{equation*}
By using the classical Hardy inequality when $\alpha=2$ and those proved by Bogdan and Dyda (\cite[Theorem~1]{BD}), we have that for every $\alpha\in (0,2]$,
\begin{equation*}Q_\lambda^\alpha(u,u)\geq 0, \quad u\in C^1_c(\hs),\end{equation*}
provided that $\lambda\geq\lambda_*$. Note that Hardy inequalities proved in \cite[Theorem~1]{BD} are sharp.

According to Friedrichs's Extension Theorem, the quadratic form $Q_\lambda^\alpha$ defines a self-adjoint and non-negative operator $L_\lambda^{\alpha}$ in $L^2(\hs)$ for which $C^1_c(\hs)$ is a form core. When $f$ is smooth enough in $\hs$, $\mathbb L_\lambda^{\alpha} f=L_\lambda^{\alpha} f$. The operator $-L_\lambda^{\alpha}$ generates in $L^2(\hs)$ a contractive analytic $C_0$-semigroup of angle $\frac\pi 2$ (\cite[Corollary~9]{Sch}). We denote by $\{W_{\lambda,z}^{\alpha}\}_{\Real z>0}$ the semigroup generated by $-L_\lambda^{\alpha}$ in $L^2(\hs)$. From \cite[Theorem~3.1]{AB}, for every $z\in \mathbb C$ with $\Real z>0$, $W_{\lambda,z}^{\alpha}$ is an integral operator whose kernel, denoted by $W^{\alpha}_{\lambda,z}(x,y)$, is analytic in $\{z\in \mathbb C: \Real z>0\}$ for every $x,y\in \hs$. In Section~\ref{sec: heat kernels estimates} we establish some estimates involving the kernel $W^{\alpha}_{\lambda,z}$ that will be useful to prove the main results.

In \cite{CKS10} and \cite{CK03} we can find how the operators $L_0^{\alpha}$ appear related to censored processes and the two-sided estimates for the heat kernels associated with the semigroup generated by $-L_0^{\alpha}$ (see \cite[Theorem~1.1]{CK03} and \cite[Theorem~1.1]{CKS10}).

We will now define the variation operator. Let $\rho>0$ and suppose $\{a_t\}_{t>0}$ is a set of complex numbers. We define the $\rho$-variation $\mathcal V\left(\{a_t\}_{t>0}\right)$ as follows
\begin{equation*}\mathcal V_\rho\left(\{a_t\}_{t>0}\right)=\sup_{0<t_0<\dots<t_n,\ n\in \mathbb N} \left(\sum_{j=0}^{n-1}\left|a_{t_{j+1}}-a_{t_j}\right|^\rho\right)^{1/\rho}.\end{equation*}

Assume that, for some $p\in (1,\infty)$, $T_t$ is a bounded operator on $L^p(\Omega,\mu)$ for every $t>0$, where $(\Omega,\mu)$ is a measure space. The variation operator $\mathcal V_\rho\left(\{T_t\}_{t>0}\right)$ is defined by
\begin{equation*}\mathcal V_\rho\left(\{T_t\}_{t>0}\right)(f)(x):=\mathcal V\left(\{T_t(f)(x)\}_{t>0}\right), \quad f\in L^p(\Omega,\mu).\end{equation*}

In order to have measurability of the function $\mathcal V_\rho\left(\{T_t\}_{t>0}\right)(f)$ we need to have some continuity property in the time variable for $T_t(f)$ (see the comments after \cite[Theorem~1.2]{CJRW1}). Boundedness properties of $\rho$-variation operator (also oscillation operators and $\lambda$-jump counting function) give us quantitative measures for pointwise convergence of the family $\{T_t(f)\}_{t>0}$.

L\'{e}pingle (\cite{Lep}) studied $\rho$-variations of martingales. He established strong $L^p$-inequalities and weak-type $(1,1)$ estimates for the $\rho$-variation operator when $\rho>2$. In \cite{JW} an example with $\rho=2$ shows that $\rho>2$ is the best exponent in the above estimate (see also \cite{Qi}). Variational inequalities for martingales can be seen as an extension of Doob's maximal inequalities.

Bourgain, in a series of papers (see \cite{Bo2, Bo3, Bo}) proved $L^p$-variational inequalities in ergodic settings. Bourgain's ideas were the starting point for using the variation operator in the ergodic theory (\cite{K1, K2, KMT, MS, Wi, ZK}), probability theory (\cite{AJS, MSZ, PX}) and harmonic analysis (\cite{BORSS, BCT, CJRW1, CJRW2, DPDU, DL, GT, K3,  MTX, MT, MiT, OSTTW}).

Our objective is to prove weighted $L^p$-inequalities for the $\rho$-variation operator $\mathcal V_\rho\left(\{t^k\partial_t^k W_{\lambda,t}^{\alpha}\}_{t>0}\right)$ where $k\in \mathbb N$. Our study is motivated by those ones in \cite{ABR} and \cite{N} where the Schr\"odinger operator in $\mathbb R^d$ with inverse square potentials is considered. As it was mentioned, $\{W_{\lambda,t}^{\alpha}\}_{t>0}$ can be seen as the semigroup generated by the Hardy operator type in a half space $\mathbb L_\lambda^{\alpha}$. Frank and Merz (\cite{FM23JFA}) recently have proved that the scales of homogeneous Sobolev spaces generated by $\mathbb L_\lambda^{\alpha}$ and by $(-\Delta)_{\hs}^\alpha$ are comparable.

Our first result is the following.
\begin{thm}\label{thm: unweighted case}
    Let $0<\alpha\leq 2$, $\rho>2$, $\lambda\geq 0$ and $1<p<\infty$. Then,  the $\rho$-variation operator $\mathcal V_\rho\left(\{t^k\partial_t^k W_{\lambda,t}^{\alpha}\}_{t>0}\right)$ is bounded on $L^p(\hs)$ for each $k\in \mathbb N$.
\end{thm}

For the proof of this theorem, we shall start with the case $\lambda=0$. We denote by $\{\mathbb W_t^{\alpha}\}_{t>0}$ the semigroup of operators generated by $-(-\Delta)^\alpha$ on $\mathbb R^d$. We have that
\begin{equation*}0<W_{0,t}^{\alpha}(x,y)=W_{0,t}^{\alpha}(y,x)\leq \mathbb W_t^{\alpha}(x,y), \quad x,y\in \hs, t>0.\end{equation*}
Then,
\begin{equation*}\int_{\hs} W_{0,t}^{\alpha}(x,y)\, dy=\int_{\hs} W_{0,t}^{\alpha}(x,y)\, dx\leq \int_{\mathbb R^d} \mathbb W_t^{\alpha}(x,y)\, dy=1, \quad x\in \hs, t>0.\end{equation*}
We deduce that, for every $t>0$, $W_{0,t}^{\alpha}$ is a contraction in $L^1(\hs)$ and in $L^\infty(\hs)$. According to \cite[Corollary~6.1]{LeMX} the $\rho$-variation operator $\mathcal V_\rho\left(\{t^k\partial_t^k W_{0,t}^{\alpha}\}_{t>0}\right)$ is bounded on $L^p(\hs)$ for every $1<p<\infty$ and $k\in \mathbb N$.

In a second step, we prove that the $\rho$-variation operator \begin{equation*}\mathcal V_\rho\left(\left\{t^k\partial_t^k \left(W_{\lambda,t}^{\alpha}-W_{0,t}^{\alpha}\right)\right\}_{t>0}\right)\end{equation*}
is bounded on $L^p(\hs)$ for every $1<p<\infty$ and $k\in \mathbb N$. In order to do so, we will consider operators that control this $\rho$-variation operator.

Suppose that $F:(0,\infty)\to \mathbb C$ is a differentiable function on $(0,\infty)$. For $n\in \mathbb N$, let $0<t_0<\dots<t_n$. We have that
\begin{equation*}
\left(\sum_{j=0}^{n-1}\left|F(t_{j+1})-F(t_j)\right|^\rho\right)^{1/\rho}\leq \sum_{j=0}^{n-1}\left|\int_{t_j}^{t_{j+1}}F'(s)\, ds\right|\leq \int_0^\infty |F'(s)|\, ds.
\end{equation*}
Thus,
\begin{equation*}
    \mathcal V_\rho\left(\{F(t)\}_{t>0}\right)\leq \int_0^\infty |F'(s)|\, ds.
\end{equation*}
From this inequality we deduce that 
\begin{equation}\label{eq: 1.1}
    \mathcal V_\rho\left(\left\{t^k\partial_t^k \left(W_{\lambda,t}^{\alpha}-W_{0,t}^{\alpha}\right)\right\}_{t>0}\right)(f)(x)\leq \int_{\hs} f(y) K_\lambda^{\alpha}(x,y)\, dy, \quad x\in \hs,
\end{equation}
where
\begin{equation*}
K_\lambda^{\alpha}(x,y):=\int_0^\infty \left|\partial_t\left(t^k \partial_t^k \left(W_{\lambda,t}^{\alpha}(x,y)-W_{0,t}^{\alpha}(x,y)\right)\right)\right|\, dt, \quad x,y\in \hs.
\end{equation*}
For $m\in \mathbb N$, we consider the operator $T_\lambda^{\alpha, m}$ defined by
\begin{equation*}
T_\lambda^{\alpha, m}(f)(x)=\int_{\hs} T_\lambda^{\alpha, m}(x,y) f(y)\, dy, \quad x\in \hs,
\end{equation*}
with
\begin{equation*}
T_\lambda^{\alpha, m}(x,y)=\int_0^\infty \left|t^m \partial_t^{m+1} \left(W_{\lambda,t}^{\alpha}(x,y)-W_{0,t}^{\alpha}(x,y)\right)\right|\, dt, \quad x,y\in \hs.
\end{equation*}
Taking into account \eqref{eq: 1.1}, to see that $\mathcal V_\rho\left(\left\{t^k\partial_t^k \left(W_{\lambda,t}^{\alpha}-W_{0,t}^{\alpha}\right)\right\}_{t>0}\right)$ is bounded on $L^p(\hs)$ it will be sufficient to show that the operator $T_\lambda^{\alpha, m}$ has this property for $m=0$ when $k=0$, and for $m=k-1$ and $m=k$ when $k\geq 1$.

We will split the operator $T_\lambda^{\alpha, m}$ into two parts, determined by the sets
\begin{align}\label{eq: G}
    \nonumber G=&G_1\cup G_2\\
    \nonumber :=&\{(x,y,t)\in \hs\times\hs\times(0,\infty): x_d \vee y_d\leq t^{1/\alpha}\}\\ 
    &\cup \{(x,y,t)\in \hs\times\hs\times(0,\infty): x_d \vee y_d> t^{1/\alpha},\ 2|x-y|\geq x_d\wedge y_d\}
\end{align}
and
\begin{equation}\label{eq: L}
    L:=\left(\hs\times\hs\times(0,\infty)\right)\setminus G.
\end{equation}
Here, the letter $G$ stands for \textit{global} and the letter $L$ means \textit{local}. 

Hence, for $m\in \mathbb N$, we decompose $T_\lambda^{\alpha, m}$ as follows
\begin{equation*}
T_\lambda^{\alpha, m}=T_{\lambda, \loc}^{\alpha, m}+T_{\lambda, \glob}^{\alpha, m}
\end{equation*}
where
\begin{equation*}
T_{\lambda, \loc/\glob}^{\alpha, m}(f)(x)=\int_{\hs} T_{\lambda, \loc/\glob}^{\alpha, m}(x,y)f(y)\, dy, \quad x\in \hs,
\end{equation*}
being
\begin{equation*}
T_{\lambda, \loc}^{\alpha, m}(x,y)=\int_0^\infty \left|t^m \partial_t^{m+1} \left(W_{\lambda,t}^{\alpha}(x,y)-W_{0,t}^{\alpha}(x,y)\right)\right|\chi_L(x,y,t)\, dt, \quad x,y\in \hs,
\end{equation*}
and
\begin{equation*}
T_{\lambda, \glob}^{\alpha, m}(x,y)=T_\lambda^{\alpha, m}(x,y)-T_{\lambda, \loc}^{\alpha, m}(x,y), \quad x,y\in \hs.
\end{equation*}
Actually, the cancellation in $W_{\lambda,t}^{\alpha}(x,y)-W_{0,t}^{\alpha}(x,y)$ is only relevant for $T_{\lambda, \loc}^{\alpha, m}$; in the global operator this cancellation does not play any role.

In Section~\ref{sec: proof-teo1.1-alfa<2} we will prove the $L^p$-boundedness of the operators $T_{\lambda, \loc}^{\alpha, m}$ and $T_{\lambda, \glob}^{\alpha, m}$ for $\alpha\in (0,2)$, and we shall deal with the case $\alpha=2$ in Section~\ref{sec: proof-teo1.1-alfa=2}. We distinguish this two cases since the estimations for the heat kernel $W^{\alpha}_{\lambda,t}(x,y)$, $x,y\in \hs$ and~$t>0$, are different (see Propositions~\ref{prop: 2.1} and~\ref{prop: 2.2}).

We now establish the weighted $L^p$-inequalities for $\mathcal V_\rho\left(\{t^k\partial_t^k W_{\lambda,t}^{\alpha}\}_{t>0}\right)$. First of all, we recall the definitions of the Muckenhoupt and Reverse H\"older classes in $\hs$. A weight in $\hs$ is a non-negative measurable function defined on $\hs$. A weight $w$ is said to be in $A_p(\hs)$ with $1<p<\infty$ when
\begin{equation*}\sup_B \left(\frac{1}{|B|}\int_B w(x)\, dx\right)\left(\frac{1}{|B|}\int_B w^{-1/(p-1)}(x)\, dx\right)^{p-1}<\infty,\end{equation*}
where the supremum is taken over all balls $B$ in $\hs$. 

If $1<q<\infty$, we say that a weight $w$ belongs to the class $\RH_q(\hs)$ if there exists $C>0$ such that
\begin{equation*}\left(\frac{1}{|B|}\int_B w^q(x)\, dx\right)^{1/q}\leq C \frac{1}{|B|}\int_B w(x)\, dx\end{equation*}
for every ball $B$ in $\hs$. 

The $p$-Lebesgue space with weight $w$ is denoted by $L^p(\hs,w)$.
\begin{thm}\label{thm: weighted case}
    Let $0<\alpha\leq 2$, $\rho>2$, $\lambda\geq 0$, $k\in \mathbb N$ and $1<p<\infty$. Then, the $\rho$-variation operator $\mathcal V_\rho\left(\{t^k\partial_t^k W_{\lambda,t}^{\alpha}\}_{t>0}\right)$ is bounded on $L^p(\hs,w)$ provided that $w\in A_p(\hs)$.
\end{thm}

Throughout this paper $c$ and $C$ will always represent positive constants that may vary on each occurrence. The expression $a\lesssim b$ will indicate that $a\leq C b$ for some positive constant $C$, whilst $a\sim b$ means that both $a\lesssim b$ and $b\lesssim a$ hold.

\section{Heat kernel estimates}\label{sec: heat kernels estimates}

Our goal in this section is to establish estimates for the time derivatives of the heat kernel $W^{\alpha}_{\lambda,t}$ that will be useful in the sequel.

We define the exponent $\mathfrak{p}$ as in~\cite{FM23JFA}. We consider
\begin{equation*}
M_\alpha = \begin{cases}
    \alpha & \text{if } \alpha \in  (0,2), \\
    \infty & \text{if } \alpha = 2,
\end{cases}
\end{equation*}
and, for every $\alpha \in (0,2]$, we define
\begin{align*}
  \Omega_\alpha: (-1,M_\alpha) & \longrightarrow \mathbb{R} \\
  s & \longmapsto \Omega_\alpha(s) = \frac{1}{\pi}\left(
 \Gamma(\alpha)\sin \frac{\pi \alpha}{2} + \Gamma(1+s) \Gamma(\alpha-s) \sin \frac{\pi (2s-\alpha)}{2}
  \right).
\end{align*}

We understand $\Omega_2(s) = s(s-1)$, $s\in (-1,\infty)$, and ${\Omega_1(s) = \frac{1}{\pi} (1-\pi s \cot (\pi s))}$, $s\in (-1,1)$. The main properties of the function $\Omega_\alpha$ can be found in~\cite[Appendix A]{FM23JFA}. For every $\lambda \geq \lambda_*$ there exists a unique $\mathfrak{p}\in \left[ \frac{\alpha-1}{2}, M_\alpha \right)$ such that $\Omega_\alpha(\mathfrak{p}) = \lambda$. Note that $\mathfrak{p}$ depends on $\lambda$ and $\alpha$ but we write $\mathfrak{p}$ instead of $\mathfrak{p}(\lambda,\alpha)$. We have that
\begin{enumerate}
    \item $\mathfrak{p} = \max\{\alpha-1, 0\}:=(\alpha-1)_+$, when $\lambda = 0$;
    \item $\mathfrak{p} > (\alpha-1)_+$, when $\lambda > 0$;
    \item $\mathfrak{p} = \frac{1}{2} \left(1+ \sqrt{1+4\lambda}\right)$, when $\alpha = 2$.
\end{enumerate}
For simplicity, we shall denote with $\gamma=(\alpha-1)_+$ for any $\alpha\in (0,2]$.

The following pointwise bounds for the heat kernel $W^{\alpha}_{\lambda,t}$, $t>0$, were established in~\cite[Theorem~9 and Theorem~10]{FM23JFA}.
\begin{prop}\label{prop: 2.1}\leavevmode
    \begin{enumerate}
        \item\label{item: prop 2.1 a} Let $\alpha \in (0,2)$ and $\lambda \geq 0$. Then, for every $x$, $y\in \hs$ and $t>0$, 
        \begin{equation*}
            W^{\alpha}_{\lambda,t}(x,y) \sim \left( 1 \wedge \frac{x_d}{t^{1/\alpha}}\right)^{\mathfrak{p}}
            \left( 1 \wedge \frac{y_d}{t^{1/\alpha}}\right)^{\mathfrak{p}}
            t^{-d/\alpha} \left( 1 \wedge \frac{t^{1/\alpha}}{|x-y|}\right)^{d+\alpha}.
        \end{equation*}
        \item Let $\lambda \geq -1/4$. Then, for every $x$, $y\in \hs$ and $t>0$, 
        \begin{equation*}
            W^{2}_{\lambda,t}(x,y) \asymp
             \left( 1 \wedge \frac{x_d}{t^{1/2}}\right)^{\mathfrak{p}}
            \left( 1 \wedge \frac{y_d}{t^{1/2}}\right)^{\mathfrak{p}}
            t^{-d/2} e^{-\frac{c|x-y|^2}{t}}.
        \end{equation*}
    \end{enumerate}
    Here the symbol $\asymp$ means the same as $\sim$ but admitting different values of $c$ in the exponential function in the upper and lower bounds. 
\end{prop}

Our next objective is to extend the upper bounds in Proposition~\ref{prop: 2.1} to the complex plane. Note first that from~\ref{item: prop 2.1 a} we deduce that 
\begin{align}\label{eq: 2.0}
            W^{\alpha}_{\lambda,t}(x,y) \lesssim 
            \left( \frac{x_d}{x_d + t^{1/\alpha}}\right)^{\mathfrak{p}}
            \left( \frac{y_d}{y_d + t^{1/\alpha}}\right)^{\mathfrak{p}}
            t^{-d/\alpha} \left( \frac{t^{1/\alpha}}{t^{1/\alpha} + |x-y|}\right)^{d+\alpha},
\end{align}
for every $x$, $y\in \hs$ and $t>0$.

Since the function $\phi(s) = \frac{s}{s+a}$, $s> -a$, is increasing for every $a>0$, it follows that
\begin{equation} \label{eq: 2.1}
    W^{\alpha}_{\lambda,t}(x,y) 
    \lesssim
    \left( \frac{|x|}{|x| + t^{1/\alpha}}\right)^{\mathfrak{p}}
            \left( \frac{|y|}{|y| + t^{1/\alpha}}\right)^{\mathfrak{p}}
            t^{-d/\alpha} \left( \frac{t^{1/\alpha}}{t^{1/\alpha} + |x-y|}\right)^{d+\alpha},
\end{equation}
for every $x$, $y\in \hs$ and $t>0$.

By using \cite[Proposition~2.2(a)]{BM23} we obtain that if $\alpha\in (0,2)$, $\lambda\geq 0$ and $\epsilon\in (0,1)$, for every $z\in \mathbb{C}$, $|\arg(z)|<\epsilon\pi/4$, and $x$, $y\in \hs$,
\begin{equation*}
    W^{\alpha}_{\lambda,z}(x,y) 
    \lesssim
    \left( \frac{|x|}{|x| + |z|^{1/\alpha}}\right)^{\mathfrak{p}}
            \left( \frac{|y|}{|y| + |z|^{1/\alpha}}\right)^{\mathfrak{p}}
            t^{-d/\alpha} \left( \frac{|z|^{1/\alpha}}{|z|^{1/\alpha} + |x-y|}\right)^{(d+\alpha)(1-\epsilon)}.
\end{equation*}

By using the Cauchy integral formula we deduce that, if $\alpha \in (0,2)$, $\lambda\geq 0$ and $k\in \mathbb{N}$,
\begin{equation*}
    |t^k \partial^k_t W^{\alpha}_{\lambda,t}(x,y)|
    \lesssim
    \left( \frac{|x|}{|x| + t^{1/\alpha}}\right)^{\mathfrak{p}}
            \left( \frac{|y|}{|y| + t^{1/\alpha}}\right)^{\mathfrak{p}}
            t^{-d/\alpha} \left( \frac{t^{1/\alpha}}{t^{1/\alpha} + |x-y|}\right)^{(d+\alpha)(1-\epsilon)},
\end{equation*}
for every $x$, $y\in \hs$ and $t>0$.

In a similar way we get that if $\alpha = 2$, $\lambda \geq -1/4$ and $k\in \mathbb{N}$
\begin{equation} \label{eq: 2.2}
|t^k \partial^k_t W^{2}_{\lambda,t}(x,y) |
    \lesssim
    \left( \frac{|x|}{|x| + t^{1/2}}\right)^{\mathfrak{p}}
            \left( \frac{|y|}{|y| + t^{1/2}}\right)^{\mathfrak{p}}
            t^{-d/2}  e^{-\frac{c|x-y|^2}{t}}
\end{equation}
for every $x$, $y\in \hs$ and $t>0$.

However we need to improve~\eqref{eq: 2.1} and~\eqref{eq: 2.2} by replacing $|x|$ and $|y|$ by $x_d$ and $y_d$ respectively. 
\begin{prop}\label{prop: 2.2}\leavevmode
    \begin{enumerate}
        \item\label{item: prop 2.2.a} Let $\alpha\in (0,2)$, $\lambda\geq 0$, $\epsilon\in(0,1)$ and $k\in\mathbb{N}$. For every $x$, $y\in \hs$ and $t>0$
        \begin{equation*}
            |t^k \partial^k_t W^{\alpha}_{\lambda,t}(x,y)| \lesssim 
            \left( \frac{x_d}{x_d + t^{1/\alpha}}\right)^{\mathfrak{p}}
            \left( \frac{y_d}{y_d + t^{1/\alpha}}\right)^{\mathfrak{p}}
            t^{-d/\alpha} \left( \frac{t^{1/\alpha}}{t^{1/\alpha} + |x-y|}\right)^{(d+\alpha)(1-\epsilon)}.
        \end{equation*}
        \item\label{item: prop 2.2.b} Let $\lambda \geq -1/4$ and $k \in \mathbb{N}$. For every $x$, $y\in \hs$ and $t>0$,
        \begin{equation*}
            |t^k \partial^k_t W^{2}_{\lambda,t}(x,y)| \lesssim 
            \left( \frac{x_d}{x_d + t^{1/2}}\right)^{\mathfrak{p}}
            \left( \frac{y_d}{y_d + t^{1/2}}\right)^{\mathfrak{p}}
            t^{-d/2} e^{-\frac{c|x-y|^2}{t}}.
        \end{equation*}
    \end{enumerate}
\end{prop}

\begin{proof}
    We prove~\ref{item: prop 2.2.a} first.  According to~\eqref{eq: 2.1} and proceeding as in the proof of~\cite[Proposition~3.4]{BuiD} (see also \cite[Proposition~2.2(a)]{BM23}) we deduce that
    \begin{equation}\label{eq: 2.4}
        |w(x_d,z) W^{\alpha}_{\lambda,z}(x,y) w(y_d,z)| \leq \frac{C}{|z|^{d/\alpha}},
    \end{equation}
    for $x$, $y\in\hs$ and $z\in \mathbb{C}$, $|\arg(z)|\leq \pi/4$, where
    \begin{equation*}w(s,z) =\left( \frac{s}{s + z^{1/\alpha}}\right)^{-\mathfrak{p}}, \end{equation*}
    for $s>0$ and $z\in\mathbb{C}$, $\arg(z)\leq \pi/2$. Indeed, since the operator $L^\alpha_\lambda$ is self-adjoint and non-negative, for every $s>0$, the operator $W^\alpha_{\lambda,is}$ is contractive in $L^2(\hs)$. In order to prove our claim it is sufficient to see that
    \begin{equation*}
        \| W^\alpha_{\lambda,t}\|_{L^1(\hs, w^{-1}) \hookrightarrow L^2(\hs)} + 
 \|W^\alpha_{\lambda,t} \|_{L^2(\hs) 
 \hookrightarrow wL^{\infty}(\mathbb{R}^d)} 
 \lesssim t^{-\frac{d}{2\alpha}}, \;\; t>0,
    \end{equation*}
    where $wL^{\infty}(\mathbb{R}^d) = \{f \text{ measurable in } \mathbb{R}^d: \, wf\in L^{\infty}(\mathbb{R}^d)\}$.

    Suppose $f\in L^{2}(\hs)$. According to~\eqref{eq: 2.0} and using~\cite[equation~(18)]{BuiD} we get
    \begin{equation*}
        \begin{split}
            |W^{\alpha}_{\lambda,t} (f)(x)| w(x_d,t) & 
            \lesssim
            \int_{\hs} \left(\frac{y_d}{y_d + t^{1/\alpha}}\right)^{\mathfrak{p}} \frac{t}{t^{1/\alpha}+ |x-y|^{d+\alpha}} |f(y)|dy
            \\ & \lesssim \left(
            \int_{\hs} \left| \left(\frac{|y|}{|y| + t^{1/\alpha}}\right)^{\mathfrak{p}} \frac{t}{t^{1/\alpha}+ |x-y|^{d+\alpha}} \right|^2 dy \right)^{1/2} \|f\|_{L^2(\hs)}
            \\ & \lesssim
            t^{-\frac{d}{2\alpha}} \|f\|_{L^2(\hs)}, 
        \end{split}
    \end{equation*}
for $ x\in \hs$ and  $t>0$.
    
    In a similar way we can see that
    \begin{equation*}
        \|W^{\alpha}_{\lambda,t} f\|_{L^2(\hs)}
        \lesssim  t^{-\frac{d}{2\alpha}} \|f\|_{L^1(\hs,w^{-1})},
    \end{equation*}
    for $f\in L^1(\hs,w^{-1})$ and $t>0$. We conclude that~\eqref{eq: 2.4} holds.

    On the other hand, according to~\eqref{eq: 2.0} it follows that
    \begin{equation*}
        \begin{split}
            |w(x_d,t) W^{\alpha}_{\lambda,t} (x,y) w(y_d,t)|
            & \lesssim t^{-d/\alpha} \left( \frac{t^{1/\alpha}}{t^{1/\alpha} + |x-y|}\right)^{d+\alpha}
            \\ & \lesssim t^{-d/\alpha} 
            \left( 1+ \frac{|x-y|}{t^{1/\alpha}} \right)^{-d-\alpha}
            \\ & \lesssim t^{-d/\alpha} 
            \left( 1+ \left(\frac{|x-y|^{\alpha}}{t}\right)^{1/\alpha} \right)^{-d-\alpha},
        \end{split}
    \end{equation*}
    for $x$, $y\in\mathbb{R}^d$ and $t>0$. 

    Let us define the function $b(x) = \left(1+s^{1/\alpha}\right)^{-d-\alpha}$, $s>0$. It is clear that $b$ is bounded and decreasing in $(0,\infty)$.
    Then, by~\cite[Proposition 3.3]{DR}, for each $\epsilon\in (0,1)$ and $\theta \in (0, \epsilon\pi/2)$ there exists $C>0$ such that
    \begin{equation*}
        |W^{\alpha}_{\lambda,z} (x,y)| \leq C \left|\frac{x_d}{x_d+z}\right|^{\mathfrak{p}}  \left|\frac{y_d}{y_d+z}\right|^{\mathfrak{p}} (\Real(z))^{-d/\alpha} b\left( \frac{|x-y|^{\alpha}}{|z|}\right)^{1-\epsilon},
    \end{equation*}
    for $x$, $y\in(0,\infty)$ and $z\in \mathbb{C}$ such that $|\arg(z)|\leq \theta$.

    If $z\in\mathbb{C}$ and $\arg(z)|\leq \theta$, with $0<\theta<\pi/2$, then 
    \begin{equation*}|a+z| = \sqrt{(a + \Real(z))^2 + (\textup{Im}(z))^2}
     \geq \sqrt{a^2 + |z|^2} \gtrsim a + |z|, \; a>0.\end{equation*}
    We conclude that for each $\epsilon\in (0,1)$ and $\theta \in (0,\epsilon\pi/2)$
    \begin{equation}\label{eq: 2.5}
     |W^{\alpha}_{\lambda,z} (x,y)| \lesssim  \left(\frac{x_d}{x_d+|z|}\right)^{\mathfrak{p}}  \left(\frac{y_d}{y_d+|z|}\right)^{\mathfrak{p}} |z|^{-d/\alpha} \left( \frac{|z|^{1/\alpha}}{|z|^{1/\alpha} + |x-y|}\right)^{(d+\alpha)(1-\epsilon)},
    \end{equation}
    where $z\in \mathbb{C}$ and $|\arg(z)|\leq \theta$. Note that $\Real(z)\simeq |z|$ for $z\in \mathbb{C}$ with $|\arg(z)|\leq \theta$. 

    Estimates like~\eqref{eq: 2.5} can also be proved by using the methods in~\cite[Theorem~2.1]{Me}.

    The Cauchy integral formula allow us to obtain~\ref{item: prop 2.2.a}. In order to prove~\ref{item: prop 2.2.b} we can proceed in a similar way. We can also prove it using Davies method as in the proof of~\cite[Theorem 3.4.8]{Da} and the Cauchy integral formula. 
\end{proof}

\section{Proof of Theorem~\ref{thm: unweighted case} for \mathinhead{\alpha\in (0,2)}{alpha in (0,2)}}\label{sec: proof-teo1.1-alfa<2}

\subsection{Global part}
We define, for $k\in \mathbb N$, $k\geq 1$, 
\begin{equation}\label{eq: def nucleo Jalfa}
    J_\lambda^{\alpha, k} (x,y)=\int_0^\infty \left|t^{k-1}\partial_t^k (W^{\alpha}_{\lambda,t}(x,y)-W^{\alpha}_{0,t}(x,y))\right|\chi_G(x,y,t)\, dt, \quad x,y\in \hs,
\end{equation}
where $G$ is as in \eqref{eq: G}.

By using Proposition~\ref{prop: 2.2}~\ref{item: prop 2.2.a} we get
\begin{align*}
    \left|J_\lambda^{\alpha, k} (x,y)\right|&\leq \int_0^\infty \left(\left|t^{k-1}\partial_t^k W^{\alpha}_{\lambda,t}(x,y)\right|+\left|t^{k-1}\partial_t^k W^{\alpha}_{0,t}(x,y)\right|\right)\chi_G(x,y,t)\, dt\\
    &\lesssim \int_0^\infty \frac{1}{t^{1+d/\alpha}} \left(\frac{x_d}{x_d+t^{1/\alpha}}\right)^{\gamma} \left(\frac{y_d}{y_d+t^{1/\alpha}}\right)^{\gamma}\\
    &\qquad \times\left(\frac{t^{1/\alpha}}{t^{1/\alpha}+|x-y|}\right)^{(d+\alpha)(1-\epsilon)}\chi_G(x,y,t)\, dt, \quad x,y\in \hs,
\end{align*}
where $0<\epsilon<1$ and, as we set before, $\gamma=(\alpha-1)_+$. Here, we can actually use that the power $\mathfrak{p}$ can be replaced by any non-negative number less or equal than $\mathfrak{p}$ in the estimate given in Proposition~\ref{prop: 2.2}~\ref{item: prop 2.2.a}.

We shall split the last integral over $G_1$ and $G_2$, and denote
\begin{align*}
    S_{\epsilon, j}^\alpha(x,y)&=\int_0^\infty \frac{1}{t^{1+d/\alpha}} \left(\frac{x_d}{x_d+t^{1/\alpha}}\right)^{\gamma} \left(\frac{y_d}{y_d+t^{1/\alpha}}\right)^{\gamma}\\
    &\quad \times\left(\frac{t^{1/\alpha}}{t^{1/\alpha}+|x-y|}\right)^{(d+\alpha)(1-\epsilon)}\chi_{G_j}(x,y,t)\, dt, \quad j=1,2.
\end{align*}
From the definition of $G_1$ we can write
\begin{align*}
    S_{\epsilon, 1}^\alpha &(x,y)\\
    &=\int_{(x_d\vee y_d)^\alpha}^\infty  \left(\frac{x_d}{x_d+t^{1/\alpha}}\right)^{\gamma} \left(\frac{y_d}{y_d+t^{1/\alpha}}\right)^{\gamma}\left(\frac{t^{1/\alpha}}{t^{1/\alpha}+|x-y|}\right)^{(d+\alpha)(1-\epsilon)}\, \frac{dt}{t^{1+d/\alpha}}
\end{align*}
If we suppose that $x,y\in \hs$ with $|x-y|\leq x_d\vee y_d$, we get
\begin{equation*}S_{\epsilon, 1}^\alpha (x,y)\lesssim \int_{(x_d\vee y_d)^\alpha}^\infty \frac{(x_d y_d)^{\gamma}}{t^{1+d/\alpha+2q/\alpha}}\, dt\leq C \frac{(x_d y_d)^{\gamma}}{(x_d\vee y_d)^{d+2q}}, \quad 0<\epsilon<1.\end{equation*}
On the other hand, if $|x-y|> x_d\vee y_d$, 
\begin{align*}
    S_{\epsilon, 1}^\alpha &(x,y)\\
    &=\int_{(x_d\vee y_d)^\alpha}^{|x-y|^\alpha} \left(\frac{x_d}{x_d+t^{1/\alpha}}\right)^{\gamma} \left(\frac{y_d}{y_d+t^{1/\alpha}}\right)^{\gamma}\left(\frac{t^{1/\alpha}}{t^{1/\alpha}+|x-y|}\right)^{(d+\alpha)(1-\epsilon)}\, \frac{dt}{t^{1+d/\alpha}}\\
    &\quad +\int_{|x-y|^\alpha}^\infty \left(\frac{x_d}{x_d+t^{1/\alpha}}\right)^{\gamma} \left(\frac{y_d}{y_d+t^{1/\alpha}}\right)^{\gamma}\left(\frac{t^{1/\alpha}}{t^{1/\alpha}+|x-y|}\right)^{(d+\alpha)(1-\epsilon)}\, \frac{dt}{t^{1+d/\alpha}}\\
    &\lesssim \int_{0}^{|x-y|^\alpha} \frac{(x_d y_d)^{\gamma}}{t^{1+d/\alpha+2q/\alpha}} \left(\frac{t^{1/\alpha}}{|x-y|}\right)^{(d+\alpha)(1-\epsilon)}\, dt+\int_{|x-y|^\alpha}^\infty \frac{(x_d y_d)^{\gamma}}{t^{1+d/\alpha+2q/\alpha}}\, dt\\
    &\lesssim  (x_d y_d)^{\gamma}\left(\frac{1}{|x-y|^{(d+\alpha)(1-\epsilon)}}\int_{0}^{|x-y|^\alpha} t^{\frac{(d+\alpha)(1-\epsilon)}{\alpha}-1-\frac{d}{\alpha}-\frac{2q}{\alpha}}\, dt\right.\\
    &\quad +\left. \int_{|x-y|^\alpha}^\infty  \frac{dt}{t^{1+d/\alpha+2q/\alpha}}  \right)\\
    &\lesssim \frac{(x_d y_d)^{\gamma}}{|x-y|^{d+2q}},
\end{align*}
provided that $(d+\alpha)(1-\epsilon)-d-2\gamma>0$, that is, whenever $0<\epsilon<(\alpha-2\gamma)/(d+\alpha)$. Note that this is possible since $\gamma=(\alpha-1)_+$ implies $\alpha-2\gamma>0$. 

We conclude that
\begin{equation*}S_{\epsilon, 1}^\alpha (x,y)\lesssim \frac{(x_d y_d)^{\gamma}}{(|x-y|\vee x_d \vee y_d)^{d+2\gamma}}, \quad x,y\in \hs\end{equation*}
when $0<\epsilon<(\alpha-2\gamma)/(d+\alpha)$.

For $j=2$, we take into account that $|x-y|\geq (x_d\wedge y_d)/2$, so we have
\begin{align*}
    S_{\epsilon, 2}^\alpha (x,y)&\lesssim \int_0^{(x_d\vee y_d)^\alpha} \left(\frac{x_d\wedge y_d}{x_d\wedge y_d+t^{1/\alpha}}\right)^{\gamma} \left(\frac{t^{1/\alpha}}{t^{1/\alpha}+|x-y|}\right)^{(d+\alpha)(1-\epsilon)}\, \frac{dt}{t^{1+d/\alpha}}\\
    &=\int_0^{(x_d\wedge y_d)^\alpha} \left(\frac{x_d\wedge y_d}{x_d\wedge y_d+t^{1/\alpha}}\right)^{\gamma} \left(\frac{t^{1/\alpha}}{t^{1/\alpha}+|x-y|}\right)^{(d+\alpha)(1-\epsilon)}\, \frac{dt}{t^{1+d/\alpha}}\\
    &\quad +\int_{(x_d\wedge y_d)^\alpha}^{(x_d\vee y_d)^\alpha}\left(\frac{x_d\wedge y_d}{x_d\wedge y_d+t^{1/\alpha}}\right)^{\gamma} \left(\frac{t^{1/\alpha}}{t^{1/\alpha}+|x-y|}\right)^{(d+\alpha)(1-\epsilon)}\, \frac{dt}{t^{1+d/\alpha}}\\
    &\lesssim \int_0^{(x_d\wedge y_d)^\alpha} \left(\frac{t^{1/\alpha}}{|x-y|}\right)^{(d+\alpha)(1-\epsilon)} \frac{dt}{t^{1+d/\alpha}}\\
    &\quad +\int_{(x_d\wedge y_d)^\alpha}^{(x_d\vee y_d)^\alpha} \left(\frac{x_d\wedge y_d}{t^{1/\alpha}}\right)^{\gamma} \left(\frac{t^{1/\alpha}}{|x-y|}\right)^{(d+\alpha)(1-\epsilon)} \frac{dt}{t^{1+d/\alpha}}\\
    &\lesssim \int_0^{(x_d\wedge y_d)^\alpha} \frac{t^{(d+\alpha)(1-\epsilon)/\alpha-1-d/\alpha}}{|x-y|^{(d+\alpha)(1-\epsilon)}}\, dt\\
    &\quad +(x_d\wedge y_d)^{\gamma}\int_0^{(x_d\vee y_d)^\alpha}\frac{t^{\frac{(d+\alpha)(1-\epsilon)}{\alpha}-1-\frac{d}{\alpha}-\frac{\gamma}{\alpha}}}{|x-y|^{(d+\alpha)(1-\epsilon)}}\, dt\\
    &\lesssim \frac{(x_d\wedge y_d)^{(d+\alpha)(1-\epsilon)-d}}{|x-y|^{(d+\alpha)(1-\epsilon)}} +\frac{(x_d\wedge y_d)^{\gamma} (x_d\vee y_d)^{(d+\alpha)(1-\epsilon)-d-\gamma}}{|x-y|^{(d+\alpha)(1-\epsilon)}}\\
    &\lesssim \frac{1}{|x-y|^{(d+\alpha)(1-\epsilon)}} (x_d\vee y_d)^{(d+\alpha)(1-\epsilon)-d-\gamma} (x_d\wedge y_d)^{\gamma}
\end{align*}
provided that $(d+\alpha)(1-\epsilon)-d-\gamma>0$, i.e., $0<\epsilon<(\alpha-\gamma)/(d+\alpha)$. Actually, this is possible since the restriction obtained in the case $j=1$ implies this one.

Recall that the above estimates for $S_{\epsilon, j}^\alpha$, $j=1,2$, are valid for any non-negative power less or equal to $\mathfrak{p}=\gamma$; in particular, for power zero. Thus, in this case we get
\begin{align*}
    \left|J_\lambda^{\alpha, k} (x,y)\right|&\lesssim \frac{1}{(|x-y|\vee x_d \vee y_d)^{d}}+\chi_{\left\{|x-y|\geq \frac{x_d\wedge y_d}{2}\right\}}(x,y)\frac{(x_d \vee y_d)^{(d+\alpha)(1-\epsilon)-d}}{|x-y|^{(d+\alpha)(1-\epsilon)}}
\end{align*}
whenever $0<\epsilon<\alpha/(d+\alpha)$.

Fix now $0<\epsilon<\alpha/(d+\alpha)$. We define now $\alpha'=(d+\alpha)(1-\epsilon)-d$. Hence (see \cite[p.~31]{FM23JFA}),
\begin{equation*}\frac{1}{(|x-y|\vee x_d \vee y_d)^{d}}\lesssim \frac{(|x-y|\vee x_d \vee y_d)^{\alpha'}}{(|x-y|\vee (x_d \wedge y_d))^{d+\alpha'}}\end{equation*}
and
\begin{equation*}\frac{(x_d \vee y_d)^{\alpha'}}{|x-y|^{d+\alpha'}}\chi_{\left\{|x-y|\geq \frac{x_d\wedge y_d}{2}\right\}}(x,y)\lesssim \frac{(|x-y|\vee x_d \vee y_d)^{\alpha'}}{(|x-y|\vee (x_d \wedge y_d))^{d+\alpha'}}.\end{equation*}

Therefore, by proceeding as in the proof of \cite[Proposition~19, Step~1]{FM23JFA}, we deduce that
\begin{equation*}\int_{\mathbb R^{d-1}} \left|J_\lambda^{\alpha, k} (x,y)\right|\, dy'\lesssim \frac{(x_d \vee y_d)^{\alpha'}}{(|x_d-y_d|\vee (x_d \wedge y_d))^{\alpha'+1}},\end{equation*}
being $y'=(y_1,\dots, y_{d-1})$.

By considering $w_\beta(x,y)=\left(\frac{x_d}{y_d}\right)^\beta$, $x,y\in \hs$ and $\beta>0$, we get, as in \cite[Proposition~19, Step~2]{FM23JFA}, that for any $1<p<\infty$,
\begin{align*}
    \sup_{x\in \hs} &\int_{\hs} w_\beta(x,y)^{1/p} \left|J_\lambda^{\alpha, k} (x,y)\right|\, dy <\infty, \quad 0<\beta/p<1,\\
    \sup_{x\in \hs} &\int_{\hs} w_\beta(x,y)^{-1/p'} \left|J_\lambda^{\alpha, k} (x,y)\right|\, dx <\infty, \quad 0<\beta/p'<1.
\end{align*}
According to Schur's test with weights (see, for instance, \cite[Lemma~3.3]{KMVZZ}), we can conclude that the integral operator defined by
\begin{equation*}J_\lambda^{\alpha, k} (f)(x)=\int_{\hs} J_\lambda^{\alpha, k} (x,y) f(y)\, dy, \quad x\in \hs,\end{equation*}
is bounded on $L^p(\hs)$ for every $1<p<\infty$.

\subsection{Local part}\label{sec: proof-teo1.1-alfa<2 local}

According to Duhamel formula we can write, for every $x,y\in \hs$ and $t>0$, that
\begin{align}\label{eq: duhamel}
    \nonumber D_{\lambda,t}^{\alpha, 0}(x,y)&:=\lambda \int_0^t \int_{\hs} W^{\alpha}_{0,t-s} (x,z) z_d^{-\alpha} W^{\alpha}_{\lambda,s} (z,y)\, dzds\\
    \nonumber &=\lambda \int_0^{t/2} \int_{\hs} W^{\alpha}_{0,t-s} (x,z) z_d^{-\alpha} W^{\alpha}_{\lambda,s} (z,y)\, dzds\\
    &\quad +\lambda \int_0^{t/2} \int_{\hs} W^{\alpha}_{0,s} (x,z) z_d^{-\alpha} W^{\alpha}_{\lambda,t-s} (z,y)\, dzds. 
\end{align}
We have that $D_{\lambda,t}^{\alpha, 0}(x,y)=t^{-d/\alpha}F\left(t^{-1/\alpha}x, t^{-1/\alpha}y\right)$, $x,y\in\hs$, $t>0$ (see comments after \cite[Theorem~9]{FM23JFA}) for certain smooth function $F:\hs\times \hs \to \mathbb R$. Thus, for every $k\in \mathbb N$, there exists an smooth function $F_k:\hs\times \hs \to \mathbb R$ for which
\begin{equation}\label{eq: P1}
    t^k\partial_t^k D_{\lambda,t}^{\alpha, 0}(x,y)=t^{-d/\alpha}F_k\left(t^{-1/\alpha}x, t^{-1/\alpha}y\right), \quad x,y\in\hs, t>0.
\end{equation}
Given $k\in \mathbb N$, from \eqref{eq: duhamel} we can write
\begin{align*}
    \partial_t^k D_{\lambda,t}^{\alpha, 0}(x,y)&=\sum_{j=0}^{k-1}a_{j,k} \int_{\hs} \partial_t^j W^{\alpha}_{0,t/2} (x,z) z_d^{-\alpha} \partial_t^{k-1-j} W^{\alpha}_{\lambda,t/2} (z,y) \, dz\\
    &\quad +\lambda \int_0^{t/2} \int_{\hs} \partial_t^k W^{\alpha}_{0,t-s} (x,z) z_d^{-\alpha} W^{\alpha}_{\lambda,s} (z,y)\, dzds\\
    &\quad +\lambda \int_0^{t/2} \int_{\hs} W^{\alpha}_{0,s} (x,z) z_d^{-\alpha} \partial_t^k W^{\alpha}_{\lambda,t-s} (z,y)\, dzds, \quad x,y\in\hs, t>0,
\end{align*}
where $a_{j,k}\in \mathbb R$ for each $j=1,\dots, k-1$.

Assume now that $x,y\in \hs$ with $|x-y|\leq (x_d\wedge y_d)/2$ and $0<t^{1/\alpha}\leq x_d\vee y_d$. Taking into account the scaling property \eqref{eq: P1}, we are going to estimate
\begin{equation*}
\left. \partial_t^k D_{\lambda,t}^{\alpha, 0}(x,y)\right|_{t=1}.
\end{equation*}

We define, for every $j=1,\dots, k-1$, the integrals involved in the above summation as
\begin{equation*}
H_j(x,y,t)=\int_{\hs} \partial_t^j W^{\alpha}_{0,t/2} (x,z) z_d^{-\alpha} \partial_t^{k-1-j} W^{\alpha}_{\lambda,t/2} (z,y) \, dz, \quad x,y\in\hs, t>0,
\end{equation*}
and we decompose each one of them as follows
\begin{align*}
    H_j(x,y,t)&=\int_{\{z\in\hs: z_d>x_d/2\}}\partial_t^j W^{\alpha}_{0,t/2} (x,z) z_d^{-\alpha} \partial_t^{k-1-j} W^{\alpha}_{\lambda,t/2} (z,y) \, dz\\
    &\quad +\int_{\{z\in\hs: z_d\leq x_d/2\}}\partial_t^j W^{\alpha}_{0,t/2} (x,z) z_d^{-\alpha} \partial_t^{k-1-j} W^{\alpha}_{\lambda,t/2} (z,y) \, dz\\
    &:=H_{j,1}(x,y,t)+H_{j,2}(x,y,t), \quad x,y\in\hs, t>0.
\end{align*}

We notice first that, since $x_d\sim y_d$ (because $|x-y|\leq (x_d\wedge y_d)/2$), when $t=1$ we have $x_d\sim y_d\gtrsim 1$.

By using Proposition~\ref{prop: 2.2}\ref{item: prop 2.2.a}, for $\epsilon\in (0,1)$ we have
\begin{align*}
    |H_{j,1}(x,y,1)|&\lesssim x_d^{-\alpha} \int_{\hs} \left(\frac{x_d}{x_d+1}\right)^{(\alpha-1)_+}\left(\frac{z_d}{z_d+1}\right)^{(\alpha-1)_+}\left(\frac{1}{1+|x-z|}\right)^{(d+\alpha)(1-\epsilon)}\\
    &\quad \times\left(\frac{y_d}{y_d+1}\right)^{(\alpha-1)_+}\left(\frac{z_d}{z_d+1}\right)^{(\alpha-1)_+}\left(\frac{1}{1+|y-z|}\right)^{(d+\alpha)(1-\epsilon)}\, dz.
\end{align*}

We define $\alpha'=\alpha(1-\epsilon)-d\epsilon$ and consider $0<\epsilon<\alpha/(\alpha+d)$. Then, $0<\alpha'<\alpha<2$, and we get
\begin{align*}
    |H_{j,1}(x,y,1)|&\lesssim x_d^{-\alpha} \int_{\hs} \left(\frac{x_d}{x_d+1}\right)^{(\alpha'-1)_+}\left(\frac{z_d}{z_d+1}\right)^{2(\alpha'-1)_+}\left(\frac{y_d}{y_d+1}\right)^{(\alpha'-1)}\\
    &\quad \times\left(\frac{1}{1+|x-z|}\right)^{d+\alpha'}\left(\frac{1}{1+|y-z|}\right)^{d+\alpha'}\, dz.    
\end{align*}
According to \cite[Theorem~9]{FM23JFA}, we have
\begin{align*}
    |H_{j,1}(x,y,1)|&\lesssim x_d^{-\alpha} \int_{\hs} W^{\alpha'}_{0,1}(x,z)W^{\alpha'}_{0,1}(z,y)\, dz\\
    &=Cx_d^{-\alpha} W^{\alpha'}_{0,2}(x,y)\\
    &\lesssim x_d^{-\alpha} \left(\frac{x_d}{x_d+1}\right)^{(\alpha'-1)_+}\left(\frac{y_d}{y_d+1}\right)^{(\alpha'-1)}\left(\frac{1}{1+|x-y|}\right)^{d+\alpha'}.
\end{align*}
We now estimate $H_{j,2}$. If $z_d\leq x_d/2$, since $|x-y|<(x_d\wedge y_d)/2$, we deduce that $|y_d-z_d|\sim y_d\sim x_d$ and $|x_d-z_d|\sim x_d$. Then, if $z_d\leq x_d/2$, we have $|x-z|\sim |x'-z'|+x_d$ and $|y-z|\sim |y'-z'|+x_d$. By \cite[Lemma~22]{FM23JFA} with $N=d-1\geq 1$, $\beta=\alpha(1-\epsilon)+1-d\epsilon$ and $r=s=1+x_d$ we get that whenever $z_d\leq x_d/2$,
\begin{align*}
    \int_{\mathbb R^{d-1}} & \left|\left. \partial_t^j W^{\alpha}_{0,t/2}(x,z)\right|_{t=1}\right|\left|\left. \partial_t^{k-1-j} W^{\alpha}_{\lambda,t/2}(z,y)\right|_{t=1}\right| dz'\\
    &\lesssim \left(\frac{x_d}{x_d+1}\right)^{(\alpha-1)_+}\left(\frac{z_d}{z_d+1}\right)^{(\alpha-1)_++\mathfrak{p}}\left(\frac{y_d}{y_d+1}\right)^{\mathfrak{p}} \\
    &\quad \times \int_{\mathbb R^{d-1}} \left(\frac{1}{1+|x'-z'|+x_d}\right)^{(d+\alpha)(1-\epsilon)}\left(\frac{1}{1+|y'-z'|+x_d}\right)^{(d+\alpha)(1-\epsilon)} dz'\\
    &\lesssim \left(\frac{x_d}{x_d+1}\right)^{(\alpha-1)_+}\left(\frac{z_d}{z_d+1}\right)^{(\alpha-1)_++\mathfrak{p}}\left(\frac{y_d}{y_d+1}\right)^{\mathfrak{p}} \frac{1}{x_d^{\alpha(1-\epsilon)+1-d\epsilon}}\\
    &\quad \times \frac{1}{(1+x_d+|y'-x'|)^{(d+\alpha)(1-\epsilon)}}\\
    &\lesssim \left(\frac{x_d}{x_d+1}\right)^{(\alpha-1)_+}\left(\frac{z_d}{z_d+1}\right)^{(\alpha-1)_++\mathfrak{p}}\left(\frac{y_d}{y_d+1}\right)^{\mathfrak{p}} \frac{1}{x_d^{\alpha(1-\epsilon)+1-d\epsilon+(d+\alpha)(1-\epsilon)}}.
\end{align*}
From \cite[p.~26]{FM23JFA} we also have that
\begin{equation*}\int_0^{x_d/2} \left(\frac{z_d}{z_d+1}\right)^{(\alpha-1)_++\mathfrak{p}} \frac{dz_d}{z_d^\alpha}\sim \boldsymbol{1}_{\alpha\geq 1}+(\ln(1+x_d))\boldsymbol{1}_{\alpha=1}+x_d^{1-\alpha}\boldsymbol{1}_{\alpha< 1}.\end{equation*}
Then, \refstepcounter{BDQ}\label{pag: cota xd>=1}
\begin{align*}
   &\frac{1}{x_d^{\alpha(1-\epsilon)+1-d\epsilon+(d+\alpha)(1-\epsilon)}} \int_0^{x_d/2} \left(\frac{z_d}{z_d+1}\right)^{(\alpha-1)_++\mathfrak{p}} \frac{dz_d}{z_d^\alpha}\\
   &\lesssim \frac{1}{x_d^{\alpha+(d+\alpha)(1-\epsilon)}} \frac{1}{x_d^{1-(d+\alpha)\epsilon}} \left(\boldsymbol{1}_{\alpha\geq 1}+(\ln(1+x_d))\boldsymbol{1}_{\alpha=1}+x_d^{1-\alpha}\boldsymbol{1}_{\alpha< 1}\right)\\
   &\lesssim \frac{1}{x_d^{\alpha+(d+\alpha)(1-\epsilon)}}\lesssim \frac{1}{x_d^\alpha} \left(\frac{1}{1+|x-y|}\right)^{(d+\alpha)(1-\epsilon)},
\end{align*}
provided that $0<\epsilon<\alpha/(d+\alpha)$ and $x_d\geq 1$. 

We consider now 
\begin{align*}
    H_k(x,y,t)&=\int_0^{t/2}\int_{\{z\in\hs: z_d>x_d/2\}}\partial_t^k W^{\alpha}_{0,t-s} (x,z) z_d^{-\alpha} W^{\alpha}_{\lambda,s} (z,y) \, dz ds\\
    &\quad +\int_0^{t/2}\int_{\{z\in\hs: z_d\leq x_d/2\}}\partial_t^k W^{\alpha}_{0,t-s} (x,z) z_d^{-\alpha} W^{\alpha}_{\lambda,s} (z,y) \, dz ds\\
    &:=H_{k,1}(x,y,t)+H_{k,2}(x,y,t), \quad x,y\in\hs, t>0.
\end{align*}
First, we estimate both terms using \eqref{eq: 2.0} and Proposition~\ref{prop: 2.2}\ref{item: prop 2.2.a}. Recalling that $\mathfrak{p}\geq (\alpha-1)_+$, for $\ell=1,2$ we get
\begin{align}\label{eq: estimacion Hkl}
    |& H_{k,\ell}(x,y,1)|\nonumber\\
    &\lesssim \int_0^{1/2} \int_{\hs} (1-s)^{-k-d/\alpha} \left(\frac{x_d}{(1-s)^{1/\alpha}+x_d}\right)^{(\alpha-1)_+} \left(\frac{z_d}{(1-s)^{1/\alpha}+z_d}\right)^{(\alpha-1)_+}\nonumber\\
    &\quad \times \left(\frac{(1-s)^{1/\alpha}}{(1-s)^{1/\alpha}+|x-z|}\right)^{(d+\alpha)(1-\epsilon)} z_d^{-\alpha} \left(\frac{y_d}{s^{1/\alpha}+y_d}\right)^{\mathfrak{p}}\left(\frac{z_d}{s^{1/\alpha}+z_d}\right)^{\mathfrak{p}}\frac{1}{s^{d/\alpha}} \nonumber\\
    &\quad \times \left(\frac{s^{1/\alpha}}{s^{1/\alpha}+|z-y|}\right)^{(d+\alpha)(1-\epsilon)}\,dzds. 
\end{align}
We define, as before, $\alpha'=\alpha(1-\epsilon)-d\epsilon$ for $0<\epsilon<\alpha/(d+\alpha)$, so $0<\alpha'<\alpha<2$. Thus, from Proposition~\ref{prop: 2.1}\ref{item: prop 2.1 a} with $\lambda=0$, we have
\begin{align*}
    &|H_{k,1}(x,y,1)|\lesssim \frac{1}{x_d^\alpha} \int_0^{1/2} \int_{\hs} \left(\frac{x_d}{(1-s)^{1/\alpha}+x_d}\right)^{(\alpha'-1)_+} \left(\frac{z_d}{(1-s)^{1/\alpha}+z_d}\right)^{(\alpha'-1)_+} \\
    &\quad \times\frac{1}{s^{d/\alpha}} \left(\frac{y_d}{s^{1/\alpha}+y_d}\right)^{(\alpha'-1)_+} \left(\frac{z_d}{s^{1/\alpha}+z_d}\right)^{(\alpha'-1)_+}\\
    &\quad\times \left(\frac{(1-s)^{1/\alpha}}{(1-s)^{1/\alpha}+|x-z|}\right)^{d+\alpha'}\left(\frac{s^{1/\alpha}}{s^{1/\alpha}+|z-y|}\right)^{d+\alpha'}\, dzds\\
    &\lesssim \frac{1}{x_d^\alpha} \int_0^{1/2} \int_{\hs} \left(\frac{x_d}{((1-s)^{\alpha'/\alpha})^{1/\alpha'}+x_d}\right)^{(\alpha'-1)_+} \left(\frac{z_d}{((1-s)^{\alpha'/\alpha})^{1/\alpha'}+z_d}\right)^{(\alpha'-1)_+} \\
    &\quad \times\frac{1}{(s^{\alpha'/\alpha})^{d/\alpha'}}\left(\frac{y_d}{(s^{\alpha'/\alpha})^{1/\alpha'}+y_d}\right)^{(\alpha'-1)_+}  \left(\frac{z_d}{(s^{\alpha'/\alpha})^{1/\alpha'}+z_d}\right)^{(\alpha'-1)_+}\\
    &\quad \times \left(\frac{((1-s)^{\alpha'/\alpha})^{1/\alpha'}}{((1-s)^{\alpha'/\alpha})^{1/\alpha'}+|x-z|}\right)^{d+\alpha'}\left(\frac{(s^{\alpha'/\alpha})^{1/\alpha'}}{(s^{\alpha'/\alpha})^{1/\alpha'}+|z-y|}\right)^{d+\alpha'}\, dzds\\
    &\sim \frac{1}{x_d^\alpha} \int_0^{1/2} W^{\alpha'}_{0,(1-s)^{\alpha'/\alpha}}(x,z) W^{\alpha'}_{0,s^{\alpha'/\alpha}}(z,y)\, dzds\\
    & \lesssim \frac{1}{x_d^\alpha} \int_0^{1/2} W^{\alpha'}_{0,(1-s)^{\alpha'/\alpha}+s^{\alpha'/\alpha}}(x,y)\, dzds\\
    &\lesssim \frac{1}{x_d^\alpha} \int_0^{1/2} \left(\frac{x_d}{((1-s)^{\alpha'/\alpha}+s^{\alpha'/\alpha})^{1/\alpha'}+x_d}\right)^{(\alpha'-1)_+} \\
    &\qquad \times\left(\frac{y_d}{((1-s)^{\alpha'/\alpha}+s^{\alpha'/\alpha})^{1/\alpha'}+y_d}\right)^{(\alpha'-1)_+}\\
    &\qquad \times\left(\frac{((1-s)^{\alpha'/\alpha}+s^{\alpha'/\alpha})^{1/\alpha'}}{((1-s)^{\alpha'/\alpha}+s^{\alpha'/\alpha})^{1/\alpha'}+|x-y|}\right)^{\alpha'+d}\, ds\\
    &\lesssim \frac{1}{x_d^\alpha} \int_0^{1/2} \left(\frac{x_d}{2^{-1/\alpha}+x_d}\right)^{(\alpha'-1)_+} \left(\frac{y_d}{2^{-1/\alpha}+y_d}\right)^{(\alpha'-1)_+} \left(\frac{2^{1/\alpha}}{2^{1/\alpha}+|x-y|}\right)^{\alpha'+d}\, ds\\
    &\lesssim \frac{1}{x_d^\alpha} \left(\frac{x_d}{1+x_d}\right)^{(\alpha'-1)_+} \left(\frac{y_d}{1+y_d}\right)^{(\alpha'-1)_+} \left(\frac{1}{1+|x-y|}\right)^{\alpha'+d}.
\end{align*}

In order to estimate $H_{k,2}$, we notice first that from \cite[Lemma~22]{FM23JFA},
\begin{align*}
    \int_{\mathbb R^{d-1}} & \left(\frac{(1-s)^{1/\alpha}}{(1-s)^{1/\alpha}+|x-z|}\right)^{(d+\alpha)(1-\epsilon)} \left(\frac{s^{1/\alpha}}{s^{1/\alpha}+|z-y|}\right)^{(d+\alpha)(1-\epsilon)}\,dz'\\
    &\lesssim \int_{\mathbb R^{d-1}} \left(\frac{(1-s)^{1/\alpha}}{(1-s)^{1/\alpha}+x_d+|x'-z'|}\right)^{(d+\alpha)(1-\epsilon)} \\
    &\qquad \times\left(\frac{s^{1/\alpha}}{s^{1/\alpha}+x_d+|z'-y'|}\right)^{(d+\alpha)(1-\epsilon)}\,dz'\\
    &\lesssim \frac{((1-s)s)^{(d+\alpha)(1-\epsilon)/\alpha}}{x_d^{\alpha(1-\epsilon)-d\epsilon+1}}\frac{1}{(x_d+|x'-y'|)^{(d+\alpha)(1-\epsilon)}},
\end{align*}
provided that $z_d\leq x_d/2$.

Then, by taking $0<\epsilon<(\alpha-\mathfrak{p})/(d+\alpha)$ (which is possible since $\mathfrak{p}< \alpha$), from \eqref{eq: estimacion Hkl} we get
\begin{align*}
    |H_{k,2}(x,y,1)|&\lesssim \int_0^{1/2} \int_0^{x_d/2} \left(\frac{x_d}{(1-s)^{1/\alpha}+x_d}\right)^{(\alpha-1)_+} \left(\frac{z_d}{(1-s)^{1/\alpha}+z_d}\right)^{(\alpha-1)_+} \\
    &\quad \times z_d^{-\alpha} \left(\frac{y_d}{s^{1/\alpha}+y_d}\right)^{\mathfrak{p}}\left(\frac{z_d}{s^{1/\alpha}+z_d}\right)^{\mathfrak{p}}\frac{1}{s^{d/\alpha}}\\
    &\quad \times\frac{((1-s)s)^{(d+\alpha)(1-\epsilon)/\alpha}}{x_d^{\alpha(1-\epsilon)-d\epsilon+1}}\frac{1}{(x_d+|x'-y'|)^{(d+\alpha)(1-\epsilon)}}\, dz_d ds\\
    &\lesssim \int_0^{x_d/2} \left(\int_0^{1/2} \left(\frac{z_d}{(1-s)^{1/\alpha}+z_d}\right)^{(\alpha-1)_+}\left(\frac{z_d}{s^{1/\alpha}+z_d}\right)^{\mathfrak{p}}\right.\\
    & \hspace{-0.55cm}\left.\phantom{\int_0^{1/2}}\times\frac{((1-s)s)^{(d+\alpha)(1-\epsilon)/\alpha} s^{-d/\alpha}}{(x_d+|x'-y'|)^{(d+\alpha)(1-\epsilon)}} ds\right)\frac{z_d^{-\alpha}}{x_d^{\alpha(1-\epsilon)-d\epsilon+1}} dz_d\\
    &\lesssim \int_0^{x_d/2} \frac{z_d^{-\alpha}}{x_d^{\alpha(1-\epsilon)-d\epsilon+1+(d+\alpha)(1-\epsilon)}} (1\wedge z_d)^{\mathfrak{p}+(\alpha-1)_+}\, dz_d.
\end{align*}
Here, we have used that, for the values of $\epsilon$ indicated above, 
\begin{align*}
    s^{(d+\alpha)(1-\epsilon)/\alpha-d/\alpha} \left(1\wedge \frac{z_d}{s^{1/\alpha}}\right)^{\mathfrak{p}} &= \left(s^{\left(1-\epsilon(d+\alpha)/\alpha\right)/\mathfrak{p}}\wedge \left(s^{\left(1-\epsilon(d+\alpha)/\alpha\right)/\mathfrak{p}-1/\alpha}\right)z_d\right)^{\mathfrak{p}}\\
    &\lesssim \left(1\wedge s^{(\alpha-\mathfrak{p}-\epsilon(d+\alpha))/(\alpha\mathfrak{p})}z_d\right)^{\mathfrak{p}}\leq (1\wedge z_d)^{\mathfrak{p}}.
\end{align*}
Finally, since $0<\epsilon< \alpha/(d+\alpha)$, by proceeding as in page~\pageref{pag: cota xd>=1} we get
\begin{equation*}|H_{k,2}(x,y,1)|\lesssim \frac{1}{x_d^\alpha} \left(\frac{1}{1+|x-y|}\right)^{(d+\alpha)(1-\epsilon)}, \quad x_d\geq 1.\end{equation*}

In a similar way, we can also see that
\begin{equation*}|\widetilde{H}_{k}(x,y,1)|\lesssim \frac{1}{x_d^\alpha} \left(\frac{1}{1+|x-y|}\right)^{(d+\alpha)(1-\epsilon)}, \quad x_d\geq 1,\end{equation*}
where
\begin{equation*}\widetilde{H}_{k}(x,y,t)=\int_0^{t/2}\int_{\hs} W^{\alpha}_{0,s} (x,z) z_d^{-\alpha} \partial_t^k W^{\alpha}_{\lambda,t-s} (z,y) \, dz ds.\end{equation*}

By putting together all of the above estimates, and using the scaling property \eqref{eq: P1}, we obtain
\begin{equation*}\left|t^{k-1} \partial_t^k D_{\lambda, t}^{\alpha, 0}(x,y)\right|\lesssim \frac{t^{-d/\alpha}}{(x_d\vee y_d)^\alpha} \left(\frac{t^{1/\alpha}}{t^{1/\alpha}+|x-y|}\right)^{(d+\alpha)(1-\epsilon)},\end{equation*}
for $0<\epsilon<(\alpha-\mathfrak{p})/(d+\alpha)$.

Consequently, recalling the definition of the local part $L$ in \eqref{eq: L}, for certain constant $C>0$ we have
\begin{align*}
    &\int_{\hs}  \int_0^\infty \left|t^{k-1} \partial_t^k D_{\lambda, t}^{\alpha, 0}(x,y)\right|\chi_{L} (x,y,t) \, dt dy\\
    &\lesssim \int_{\{y\in \hs: C^{-1}x_d\leq y_d\leq Cx_d\}} \int_0^{(x_d\vee y_d)^\alpha} \frac{t^{-d/\alpha}}{(x_d\vee y_d)^\alpha} \left(\frac{t^{1/\alpha}}{t^{1/\alpha}+|x-y|}\right)^{(d+\alpha)(1-\epsilon)} \, dt dy\\
    &\lesssim \int_0^{(Cx_d)^\alpha} t^{-d/\alpha} \int_{\{y\in \hs: C^{-1}x_d\leq y_d\leq Cx_d\}} x_d^{-\alpha}\left(\frac{t^{1/\alpha}}{t^{1/\alpha}+|x-y|}\right)^{(d+\alpha)(1-\epsilon)} \, dy dt\\
    &\lesssim x_d^{-\alpha}\int_0^{(Cx_d)^\alpha} t^{-d/\alpha} \int_{\mathbb R^d} \left(\frac{t^{1/\alpha}}{t^{1/\alpha}+|x-y|}\right)^{(d+\alpha)(1-\epsilon)} \, dy dt\\
    &\lesssim x_d^{-\alpha}\int_0^{(Cx_d)^\alpha} \int_{\mathbb R^d} \frac{1}{(1+|z|)^{(d+\alpha)(1-\epsilon)}} \, dy dt\\
    &\lesssim 1,
\end{align*}
for each $x\in \hs$, provided that $0<\epsilon<(\alpha-\mathfrak{p})/(d+\alpha)$.

By symmetry, we also deduce that
\begin{equation*}\int_{\hs}  \int_0^\infty \left|t^{k-1} \partial_t^k D_{\lambda, t}^{\alpha, 0}(x,y)\right|\chi_{L} (x,y,t) \, dt dx\lesssim 1, \quad y\in \hs,\end{equation*}
when $0<\epsilon<\alpha/(d+\alpha)$.

Schur's test allows us to conclude thet the integral operator
\begin{equation*}\mathbb J_\lambda^{\alpha, k}(f)(x):=\int_{\hs}  \int_0^\infty \left|t^{k-1} \partial_t^k D_{\lambda, t}^{\alpha, 0}(x,y)\right|\chi_{L} (x,y,t) f(y)\, dt dy, \quad x\in \hs,\end{equation*}
is bounded on $L^p(\hs)$ for every $1\leq p\leq \infty$.

\section{Proof of Theorem~\ref{thm: unweighted case} for \mathinhead{\alpha=2}{alpha=2}}\label{sec: proof-teo1.1-alfa=2}

\subsection{Global part}

For the case $\alpha=2$ we know, by Proposition~\ref{prop: 2.2}\ref{item: prop 2.2.b}, that
\begin{align*}
    \left|t^k \partial_t^k W^{2}_{\lambda,t}(x,y)\right|&\lesssim \left( 1\wedge \frac{x_d}{t^{1/2}}\right)^{\mathfrak{p}} \left( 1+\frac{y_d}{t^{1/2}}\right)^{\mathfrak{p}} \frac{1}{t^{d/2}} e^{-\frac{c|x-y|^2}{t}}\\
    & \lesssim \left( 1\wedge \frac{x_d}{t^{1/2}}\right)^{\mathfrak{p}} \left( 1+\frac{y_d}{t^{1/2}}\right)^{\mathfrak{p}} \frac{1}{t^{d/2}} \left(\frac{t^{1/2}}{t^{1/2}+|x-y|}\right)^{d+2}, 
\end{align*}
for any $x,y\in \hs, t>0$. 

With this estimate, we can proceed as in the case $\alpha\in (0,2)$ (see \S\ref{sec: proof-teo1.1-alfa<2}) and prove that, for every $k\in \mathbb N$, $k\geq 1$, the operator $J_\lambda^{2, k}$ given by
\begin{equation*}J_\lambda^{2, k} (f)(x)=\int_{\hs} J_\lambda^{2, k} (x,y) f(y)\, dy, \quad x\in \hs,\end{equation*}
whose kernel $J_\lambda^{2, k} (x,y)$ is as in \eqref{eq: def nucleo Jalfa} with $\alpha=2$, 
is bounded on $L^p(\hs)$ for every $1<p<\infty$.

\subsection{Local part}

When dealing with the local part and $\alpha=2$, we keep the notation used in Section~\ref{sec: proof-teo1.1-alfa<2 local}. We recall that in the local part, we are considering $x,y\in \hs$ with $|x-y|\leq (x_d\wedge y_d)/2$ and $0<t^{1/\alpha}\leq x_d\vee y_d$.

Let $j=1,\dots, k-1$. According to Proposition~\ref{prop: 2.2}\ref{item: prop 2.2.b} we get the following estimate 
\begin{align*}
    |H_{j,1}(x,y,1)|&\lesssim x_d^{-2} \int_{\hs} \left(\frac{x_d}{1+x_d}\right)\left(\frac{z_d}{1+z_d}\right)^{1+\mathfrak{p}}\left(\frac{y_d}{1+y_d}\right)^{\mathfrak{p}} e^{-c(|x-z|^2+|z-y|^2)} \, dz\\
    &\lesssim x_d^{-2} \int_{\hs} \left(\frac{x_d}{1+x_d}\right)\left(\frac{z_d}{1+z_d}\right)^{1+\mathfrak{p}}\left(\frac{y_d}{1+y_d}\right)^{\mathfrak{p}} e^{-\frac c2 |x-y|^2}e^{-\frac c2|z-y|^2} \, dz\\
    &\lesssim x_d^{-2}  \left(\frac{x_d}{1+x_d}\right) \left(\frac{y_d}{1+y_d}\right)^{\mathfrak{p}} e^{-\frac c2|x-y|^2}\int_{\mathbb R^d} e^{-\frac c2|x-z|^2}\, dz\\
    &\lesssim x_d^{-2}e^{-\frac c2|x-y|^2}. 
\end{align*}
Therefore, for $x,y\in \hs$ with $|x-y|\leq (x_d\wedge y_d)/2$ and $1 \leq x_d\vee y_d$, we get
\begin{equation*}|H_{j,1}(x,y,1)|\lesssim \frac{1}{x_d^2} \left(\frac{1}{1+|x-y|}\right)^{d+\alpha}.\end{equation*}
Similarly, we get
\begin{equation*}|H_{j,2}(x,y,1)|\lesssim \frac{1}{x_d^2} \left(\frac{1}{1+|x-y|}\right)^{d+\alpha}.\end{equation*}

On the other hand, Proposition~\ref{prop: 2.2}\ref{item: prop 2.2.b} leads to
\begin{align*}
    |H_{k,1}&(x,y,1)|\\
    &\lesssim \int_0^{1/2} \int_{\{z\in \hs : z_d>x_d/2\}} (1-s)^{-k-d/2}\left(\frac{x_d}{(1-s)^{1/2}+x_d}\right) \left(\frac{z_d}{(1-s)^{1/2}+z_d}\right) \nonumber\\
    &\quad \times z_d^{-2} \left(\frac{y_d}{s^{1/2}+y_d}\right)^{\mathfrak{p}}\left(\frac{z_d}{s^{1/2}+z_d}\right)^{\mathfrak{p}}\frac{1}{s^{d/2}}  e^{-c\frac{|x-z|^2}{1-s}-c\frac{|z-y|^2}{s}}\,dzds\\
    &\lesssim x_d^{-2}\int_0^{1/2} \frac{x_d}{(1-s)^{1/2}+x_d}\left(\frac{y_d}{s^{1/2}+y_d}\right)^{\mathfrak{p}} e^{-\frac c2\frac{|x-y|^2}{1-s}} \int_{\mathbb R^d} \frac{e^{-\frac c2\frac{|y-z|^2}{s}}}{s^{d/2}} \,dz ds\\
    &\lesssim x_d^{-2}\int_0^{1/2} \frac{x_d}{(1-s)^{1/2}+x_d}\left(\frac{y_d}{s^{1/2}+y_d}\right)^{\mathfrak{p}} e^{-\frac c2\frac{|x-y|^2}{1-s}}\, ds\\
    &\lesssim x_d^{-2} e^{-\frac c2|x-y|^2},
\end{align*}
for $(x,y,1)\in L$.

We now study $H_{k,2}$. Using again Proposition~\ref{prop: 2.2}\ref{item: prop 2.2.b}, it yields, as before,
\begin{align*}
    |H_{k,2}&(x,y,1)|\\
    &\lesssim \int_0^{1/2} \int_{\{z\in \hs : z_d\leq x_d/2\}} (1-s)^{-k-d/2}\left(\frac{x_d}{(1-s)^{1/2}+x_d}\right) \left(\frac{z_d}{(1-s)^{1/2}+z_d}\right) \nonumber\\
    &\quad \times z_d^{-2} \left(\frac{y_d}{s^{1/2}+y_d}\right)^{\mathfrak{p}}\left(\frac{z_d}{s^{1/2}+z_d}\right)^{\mathfrak{p}}\frac{1}{s^{d/2}}  e^{-c\frac{|x-z|^2}{1-s}-c\frac{|z-y|^2}{s}}\,dzds.
\end{align*}
Note that, for each $s\in (0,\frac12)$, we can write 
\begin{equation*}
    \int_{\mathbb R^{d-1}} e^{-c\frac{|x'-z'|^2}{1-s}-c\frac{|y'-z'|^2}{s}} dz'\lesssim \int_{\mathbb R^{d-1}} e^{-c\frac{|w|^2}{s}} dw\lesssim s^{(d-1)/2}.
\end{equation*}
Then, for $s\in (0,\frac12)$,
\begin{align*}
    &\int_{\{z\in \hs : z_d\leq x_d/2\}} \left(\frac{z_d}{(1-s)^{1/2}+z_d}\right) \left(\frac{z_d}{s^{1/2}+z_d}\right)^{\mathfrak{p}} z_d^{-2} e^{-c\frac{|x-z|^2}{1-s}-c\frac{|z-y|^2}{s}}\,dz\\
    &\lesssim \int_0^{x_d/2} \left(\frac{z_d}{(1-s)^{1/2}+z_d}\right) \left(\frac{z_d}{s^{1/2}+z_d}\right)^{\mathfrak{p}} z_d^{-2} s^{(d-1)/2} e^{-\frac c2\frac{|x-y|^2}{1-s}-c\frac{|z_d-y_d|^2}{s}}\, dz_d.
\end{align*}
Combining the above estimates, we get
\begin{align*}
    |H_{k,2}(x,y,1)|&\lesssim \int_0^{x_d/2} \int_0^{1/2} s^{(d-1)/2} \left(\frac{z_d}{(1-s)^{1/2}+z_d}\right) \left(\frac{z_d}{s^{1/2}+z_d}\right)^{\mathfrak{p}} z_d^{-2}\\
    &\quad \times\left(\frac{(1-s)^{1/2}}{(1-s)^{1/2}+x_d+|y'-x'|}\right)^{d+4}\, ds dz_d\\
    &\lesssim \int_0^{x_d/2} \int_0^{1/2} s^{(d-1)/2} \left(1\wedge \frac{z_d}{(1-s)^{1/2}}\right)\left(1\wedge \frac{z_d}{s^{1/2}}\right)^{\mathfrak{p}} \, ds \\
    &\quad \times \frac{dz_d}{z_d^2} \left(\frac{1}{x_d+|y'-x'|}\right)^{d+4}\\
    &\lesssim \int_0^{x_d/2} \int_0^{1/2} s^{(d-1)/2-1/4}(1\wedge z_d)^{1+1/4}\frac{dz_d}{z_d^2} \frac{1}{x_d^{d+4}}\\
    &\lesssim \int_0^1 z_d^{d-3/4} dz_d +\int_1^\infty \frac{dz_d}{z_d^2} \frac{1}{x_d^{d+4}}\\
    &\lesssim \frac{1}{x_d^{d+4}}\\
    &\lesssim \frac{1}{x_d^2}\left(\frac{1}{1+|x-y|}\right)^{d+2}, \quad \frac{x_d}{2}\geq 1.
\end{align*}
In a similar way, we deduce that
\begin{equation*}|\widetilde{H}_k(x,y,1)|\lesssim \frac{1}{x_d^2}\left(\frac{1}{1+|x-y|}\right)^{d+2}, \quad x_d\geq 1.\end{equation*}
We now put all of the above estimates together obtaining
\begin{equation*}\left|\left.\partial_t^k D_{\lambda}^{2,0}(x,y)\right|_{t=1}\right|\lesssim \frac{1}{x_d^2}\left(\frac{1}{1+|x-y|}\right)^{d+2}.\end{equation*}
Therefore,
\begin{equation*}\left|t^{k-1}\partial_t^k D_{\lambda}^{2,0}(x,y)\right|\lesssim \frac{1}{t^{d/2}}\left(\frac{t^{1/2}}{t^{1/2}+|x-y|}\right)^{(d+2)(1-\epsilon)} \frac{1}{(x_d\vee y_d)^2},\end{equation*}
with $0<\epsilon<(2-\mathfrak{p})/(d+2)=(\alpha-\mathfrak{p})/(d+\alpha)$.

\section{Proof of Theorem~\ref{thm: weighted case}}

We proceed as in the proof of~\cite[Proposition 3.5]{N} and of the weighted property in~\cite[Theorem 1.1]{ABR} by using~\cite[Proposition 2.3]{BB} (see also~\cite[Theorem 6.6]{BZ}). We define, for every $t>0$,
\begin{equation*}
A^\alpha_{\lambda, t} = (I - W^\alpha_{\lambda,t})^m,
\end{equation*}
for $m\in \mathbb{N}$ to be fixed later. 

For every $x\in\mathbb{R}^d$, $r$, $s>0$, let $B(x,r) = \{y\in \mathbb{R}^d: |x-y|<r\}$ be the usual ball in $\mathbb{R}^d$, $sB(x,r) = B(x,sr)$ and $\mathcal{B}(x,r) = B(x,r)\cap \hs$. Also, let us consider for a ball $B\subset\mathbb{R}^d$, the sets $S_0(B) = B$ and for $j\in\mathbb{N}$, $j\geq 1$, $S_j(B) = 2^{j}B\setminus 2^{j-1}B$. We denote also, for $B\subset\mathbb{R}^d$ and $j\in\mathbb{N}$,  $S_j(\mathcal{B}) = S_j(B) \cap \hs$. 

From now on, we will assume that $1<p\leq q<\infty$. Suppose that $\mathcal{B}=B\cap \hs$, where $B = B(x_B,r_B)$ with $x_B\in\hs$, $r_B>0$ and that $f$ is a smooth function supported in $\mathcal{B}$. 

We consider first the case $\alpha = 2$. By using~\cite[(2.17)]{ABR}, from Proposition~\ref{prop: 2.2}~\ref{item: prop 2.2.b} we deduce
\begin{equation*}
    \left(\frac{1}{|S_j(\mathcal{B})|}\int_{S_j(\mathcal{B})}
    |A^2_{\lambda,r^2_B}(f)(x)|^q dx
    \right)^{1/q}
    \leq C \frac{e^{-c2^{2j}}}{2^{j\alpha/q}}\left( \frac{1}{|\mathcal{B}|} \int_\mathcal{B} |f(x)|^p dx
    \right)^{1/p},
\end{equation*}
for $j\in\mathbb{N}$ where $C>0$ does not depend on $j$, $\beta$ or $f$. 

On the other hand, the arguments in~\cite[pages 13 and 14]{ABR} allow us to obtain that
\begin{align*}
        & \left(\frac{1}{|S_j(\mathcal{B})|}\int_{S_j(\mathcal{B})}
    |V_\rho(\{t^k \partial^k_t W^w_{\lambda,t}\}_{t>0})(I - A^2_{\lambda,r^2_B})(f)(x)|^q dx
    \right)^{1/q}&\lesssim 2^{-j(dq + 2m)} f_{q, \mathcal{B}},
    \end{align*}
where $f_{q,\mathcal{B}}=\left( \frac{1}{|\mathcal{B}|}\int_{\mathcal{B}} |f(x)|^q dx\right)^{1/q}$ and the constant does not depend on $j$, $\mathcal{B}$ or $f$.

We choose now $m\in\mathbb{N}$ such that $m>d/2$. Then
    \begin{equation*}\sum_{j=1}^{\infty} 2^{-j(d/q +2m -d)}<\infty.\end{equation*}

    According to Theorem~\ref{thm: unweighted case}, the $\rho$-variation operator $V_{\rho}(\{t^k \partial^k_t W^2_{\lambda,t}\}_{t>0})$ is bound\-ed on $L^q(\hs)$. By using \cite[Proposition 2.3]{BB} we conclude that $V_{\rho}(\{t^k \partial^k_t W^2_{\lambda,t}\}_{t>0})$ is bounded on $L^{r}(\hs,w)$, provided that $p<r<q$ and $w\in A_{r/p} (\hs)\cap \textup{RH}_{(q/r)'}(\hs)$. Taking $p=1$ we obtain that $V_{\rho}(\{t^k \partial^k_t W^2_{\lambda,t}\}_{t>0})$ is bounded on $L^{r}(\hs,w)$, provided that $w\in A_r(\hs)\cap  \textup{RH}_{s'}(\hs)$, $1<r<s<\infty$. Since, if $w\in A_r(\hs)$ with ${1<r<\infty}$, there exists $\beta>1$ such that $w\in \textup{RH}_\beta(\hs)$ we conclude that $V_{\rho}(\{t^k \partial^k_t W^2_{\lambda,t}\}_{t>0})$ is bounded on $L^{r}(\hs,w)$ for every $1<r<\infty$ and $w\in A_r(\hs)$.

    Assume now $0<\alpha<2$. According to Proposition~\ref{prop: 2.2}~\ref{item: prop 2.2.a} and by proceeding as in the proof of~\cite[Theorem 5.1]{BD} we deduce that if $1\leq p<q<\infty$, $\ell\in\mathbb{N}$ and $0<\epsilon<1$, there exists $C>0$ such that for every ball $\mathcal{B} = B(x_B,r_B)\cap \hs$, with $x_B\in\hs$ and $r_B>0$, every $t>0$ and $j\in\mathbb{N}$, we have that
    \begin{align}\label{eq: 5.3}
        & \left(\frac{1}{|S_j(\mathcal{B})|}\int_{S_j(\mathcal{B})}
        |t^\ell \partial^\ell_t W^\alpha_{\lambda,t}(f)(x)|^q dx
        \right)^{1/q}\nonumber
        \\ &\leq C \max \left\{
        \left( \frac{r_B}{t^{1/\alpha}}\right)^{d/p}, \left( \frac{r_B}{t^{1/\alpha}}\right)^d
        \right\}\nonumber
        \\ & \, \times 
        \left( 1+ \frac{t^{1/\alpha}}{2^jr_B}\right)^{d/q}
        \left(1+ \frac{2^jr_B}{t^{1/\alpha}}\right)^{-(d+\alpha)(1-\epsilon)}
        \left( \frac{1}{|{\mathcal{B}}|} \int_{\mathcal{B}} |f(x)|^p dx
        \right)^{1/p},
    \end{align}
for every $f\in L^p(\hs)$ supported in $\mathcal{B}$, and 
\begin{equation}\label{eq: 5.4}
    \begin{split}
        & \left(\frac{1}{|S_j(\mathcal{B})|}\int_{S_j(\mathcal{B})}
    |t^\ell \partial^\ell_t W^\alpha_{\lambda,t}(f)(x)|^q dx
    \right)^{1/q}
    \\ &\lesssim \max \left\{
    \left( \frac{2^\ell r_B}{t^{1/\alpha}}\right)^{d/p}, \left( \frac{2^\ell r_B}{t^{1/\alpha}}\right)^d
    \right\}
    \\ & \, \times 
    \left( 1+ \frac{t^{1/\alpha}}{2^jr_B}\right)^{d/q}
    \left(1+ \frac{2^jr_B}{t^{1/\alpha}}\right)^{-(d+\alpha)(1-\epsilon)}
    \left( \frac{1}{|{S_j(\mathcal{B})}|} \int_{S_j(\mathcal{B})} |f(x)|^p dx
    \right)^{1/p},
    \end{split}
\end{equation}
for every $f\in L^p(S_j(\mathcal{B}))$.

According to~\eqref{eq: 1.1} we obtain
\begin{equation*}
    V_{\rho}\left(\{t^k \partial^k_t W^2_{\lambda,t}\}_{t>0}\right)(f)(x) \leq \int_0^\infty \left|\partial_t (t^k \partial^k_t W^2_{\lambda,t}(f)(x))\right|dt, \quad  x\in\hs.
\end{equation*}

Now, for $\ell\in\mathbb{R}\setminus \{0\}$, we consider the operator
\begin{equation*}
    \mathcal{T}_\ell (f)(x) = \int_0^\infty |t^{\ell-1} \partial_t^\ell W^\alpha_{\lambda,t} (f) (x) | dt, \; x\in\mathbb{R}^d. 
\end{equation*}

Assume that $j\in\mathbb{N}$ and let $\mathcal{B} = B(x_B,r_B)\cap \hs$, with $x_B\in\hs$ and $r_B>0$. By using Minkowski inequality we get
\begin{align*}
        &\left( \frac{1}{S_j(\mathcal{B})|} \int_{S_j(\mathcal{B})} 
        |\mathcal{T}_l ((I-W^\alpha_{\lambda,r_B^\alpha})^m(f))(x)|^q dx\right)^{1/q}
    \\ &\leq \int_0^\infty \left(
\frac{1}{S_j(\mathcal{B})|}
\int_{S_j(\mathcal{B})}
|t^{\ell-1} \partial^l_t W^\alpha_{\lambda,t} 
  ((I-W^\alpha_{\lambda,r_B^\alpha})^m(f))(x)|^q dx\right)^{1/q}dt
\\ &\leq 
\left( \int_0^{r_B^\alpha} + \int_{r_B^\alpha}^\infty \right) 
\left(
\frac{1}{S_j(\mathcal{B})|}
\int_{S_j(\mathcal{B})}
|t^{\ell-1} \partial^l_t W^\alpha_{\lambda,t} 
  ((I-W^\alpha_{\lambda,r_B^\alpha})^m(f))(x)|^q dx\right)^{1/q}dt
  \\ 
  &:= \textup{I}_1 + \textup{I}_2.
    \end{align*}

We now adapt the arguments in~\cite[p. 15 and p. 16]{N}. We begin analysing  $\textup{I}_1$. First observe that
\begin{equation*}
    (I - W^\alpha_{\lambda,r_B^\alpha})^m = 
    \int_0^{r_B^\alpha} \dots \int_0^{r_B^\alpha} \partial_{s_1} \dots\partial_{s_m} W^\alpha_{\lambda,s_1+\dots+s_m} d_{s_1\dots} d_{s_m}.
\end{equation*}

Then, by~\eqref{eq: 5.3} and calling $\mathfrak{s} = s_1+\dots +s_m$ and $d\mathfrak{s} = ds_1\dots ds_m$,  
\begin{align*}
    \textup{I}_2  & \leq \int_{r_B^\alpha}^\infty
    \int_0^{r_B^\alpha} \dots \int_0^{r_B^\alpha}\left(
    \frac{1}{|S_j(\mathcal{B})|}
    \int_{S_j(\mathcal{B})}
    |t^{\ell-1} \partial^l_t \partial_{s_1} \dots\partial_{s_m} W^\alpha_{\lambda,t +\mathfrak{s}} (f)(x)|^q dx\right)^{1/q}d\mathfrak{s}dt\\ 
    & \leq 
    \int_{r_B^\alpha}^\infty
    \int_0^{r_B^\alpha} \dots \int_0^{r_B^\alpha}
    \frac{1}{(t+ \mathfrak{s})^{m+1}}\\
    &\qquad \times
     \left(
    \frac{1}{|S_j(\mathcal{B})|}
    \int_{S_j(\mathcal{B})}
    \left|\left. u^{\ell+m} \partial^{\ell+m}_u W^\alpha_{\lambda,u} (f)(x)\right|_{u=t+\mathfrak{s}}\right|^q dx\right)^{1/q} d\mathfrak{s}dt
    \\ 
    &\lesssim
    f_{q,\mathcal{B}} \left(\int_{r_B^\alpha}^\infty
    \int_0^{r_B^\alpha} \dots \int_0^{r_B^\alpha}
    \frac{1}{(t+ \mathfrak{s})^{m+1}}
    \left( \frac{r_B}{(t+\mathfrak{s})^{1/\alpha}}\right)^{d/q}
    \left( 1+ \frac{(t+\mathfrak{s})^{1/\alpha}}{2^jr_B}\right)^{d/q}\right. \\
    &\quad \times \left.\left( 1+ \frac{2^j r_B}{(t+ \mathfrak{s})^{1/\alpha}}\right)^{-(d+\alpha)(1-\epsilon)} d\mathfrak{s} dt\right)
    \\ &\lesssim f_{q,\mathcal{B}}
    \left(\int_{r_B^\alpha}^{(2^jr_B)^\alpha} 
    \int_0^{r_B^\alpha} \dots \int_0^{r_B^\alpha}
    \frac{1}{t^{m+1}} \left(\frac{r_B}{t^{1/\alpha}}\right)^{d/q} \left( \frac{t^{1/\alpha}}{2^{j}r_B}\right)^{d/q}  d\mathfrak{s} dt\right.
    \\ 
    & \quad \left. + \int_{(2^jr_B)^\alpha}^\infty
    \int_0^{r_B^\alpha} \dots \int_0^{r_B^\alpha}
    \frac{1}{t^{m+1}} \left(\frac{r_B}{t^{1/\alpha}}\right)^{d/q} \left( \frac{t^{1/\alpha}}{2^{j}r_B}\right)^{(d+\alpha)(1-\epsilon)} d\mathfrak{s}  dt\right)\\
    &\lesssim f_{q,\mathcal{B}}\left(
    2^{-j(d+\alpha)(1-\epsilon)} \int_{r_B^\alpha}^{(2^j r_B)^{\alpha}} t^{-m-1-\frac{d}{q^\alpha}+\frac{(d+\alpha)(1-\epsilon)}{\alpha}} dt \right.\\ 
    & \quad \times \left. r_B^{\frac{d}{q} + m\alpha - (d+\alpha)(1-\epsilon)}  +\int_{(2^jr_B)^\alpha}^\infty 2^{-dj/q} r_B^{\alpha m} t^{-m-1}dt 
    \right)
    \\ &\lesssim
     f_{q,\mathcal{B}} \left(
     2^{-j(d+\alpha)(1-\epsilon)} + 2^{-j(\alpha m + d/q)}
     \right)
     \\ & \lesssim
     f_{q,\mathcal{B}} 2^{-j(d+\alpha)(1-\epsilon)},
\end{align*}
provided that $m>\frac{(d+\alpha)(1-\epsilon)}{\alpha} - \frac{d}{q^\alpha}$.
Hence, if $m>\frac{d+\alpha}{\alpha} - \frac{dq}{q^\alpha}$,
by choosing $0<\epsilon<1$ such that $m> \frac{(d+\alpha)(1-\epsilon)}{\alpha} - \frac{d}{q^\alpha} $ the estimates hold.

We can conclude that 
\begin{equation*}
    \textup{I}_2 \leq C 2^{-j(d+\alpha)(1-\epsilon)} f_{q,\mathcal{B}}.
\end{equation*}

On the other hand, we can prove by using~\eqref{eq: 5.4} (see~\cite[p. 15 and p.16]{N}) that
\begin{equation*}
    \textup{I}_1 \leq C 2^{-j(d+\alpha)(1-\epsilon)} f_{q,\mathcal{B}},
\end{equation*}
where $0<\epsilon<1$.

We obtain that
\begin{align*}
        \left( \frac{1}{|S_j(\mathcal{B})|} \int_{S_j(\mathcal{B})} |\mathcal{T}_\ell ( (I-W^\alpha_{\lambda,r_B^\alpha})^m) (f)(x)|^q dx\right)^{1/q}
        & \lesssim 2^{-j(d+\alpha)(1-\epsilon)} 
 f_{q,\mathcal{B}},
    \end{align*}
provided that $0<\epsilon<1$ and $m> \frac{(d+\alpha)(1-\epsilon)}{\alpha} - \frac{d}{q^\alpha}$.

Let $\epsilon \in (0,1)$. According to~\eqref{eq: 5.3} we deduce that 
\begin{align*}
    &\left( \frac{1}{|S_j(\mathcal{B})|} \int_{S_j(\mathcal{B})} |(I-A^\alpha_{\lambda,r^\alpha}) (f)(x)|^q \right)^{1/q}\\
    &\lesssim \sum_{i=1}^m \left( 
    \frac{1}{|S_j(\mathcal{B})|} \int_{|S_j(\mathcal{B})|} |W^{\alpha}_{\lambda,r_B^{\alpha_i}}(f)(x)|^q dx\right)^{1/q}
    \\ &\lesssim \sum_{i=1}^m \left(\frac{1}{i^{1/\alpha}}\right)^{1/p}\left(1+\frac{i^{1/\alpha}}{2^j}\right)^{\alpha/q}\left(1+\frac{2^j}{i^{1/\alpha}}\right)^{-(d+\alpha)(1-\epsilon)} \left(\frac{1}{|\mathcal{B}|}\int_{\mathcal{B}} |f(x)|^p dx\right)^{1/p}\\
    &\lesssim 2^{-j(d+\alpha)(1-\epsilon)}\left(\frac{1}{|\mathcal{B}|}\int_{\mathcal{B}} |f(x)|^p dx\right)^{1/p}.
\end{align*}
We define $\alpha_j=2^{-j(d+\alpha)(1-\epsilon)}$, $j\in \mathbb N$. We have that $\sum_{j=1}^\infty \alpha_j 2^{j\alpha}<\infty$ provided that $0<\epsilon<\alpha/(d+\alpha)$.

By \cite[Proposition~2.3]{BB} (see also \cite[Theorem~6.6]{BZ}) and by proceeding as in the case $\alpha=2$, we conclude that the operator $\mathcal{T}_\ell$ is bounded on $L^p(\hs, w)$ for every $1<p<\infty$ and $w\in A_p(\hs)$. Hence, the $\rho$-variation operator $\mathcal{V}_\rho\left(\{t^k\partial_t^kW^{\alpha}_{\lambda,t}\}_{t>0}\right)$ is also bounded on $L^p(\hs, w)$ for every $1<p<\infty$ and $w\in A_p(\hs)$.

\subsection*{Statements and Declarations} 

\subsubsection*{Funding} The first author is partially supported by grant PID2019-106093GB-I00 from the Spanish Government. The second author is partially supported by grants PICT-2019-2019-00389 (Agencia Nacional de Promoción Científica y Tecnológica), PIP-1220200101916O (Consejo Nacional de Investigaciones Científicas y Técnicas) and CAI+D 2019-015 (Universidad Nacional del Litoral)

\subsubsection*{Competing Interests} The authors have no relevant financial or non-financial interests to disclose.

\subsubsection*{Author Contributions} All authors whose names appear on the submission made substantial contributions to the conception and design of the work, drafted the work and revised it critically, approved this version to be published, and agree to be accountable for all aspects of the work in ensuring that questions related to the accuracy or integrity of any part of the work are appropriately investigated and resolved.

\subsection*{Data availability} Data sharing not applicable to this article as no datasets were generated or analysed during the current study.

\end{document}